\documentclass[11pt, twoside, leqno]{article}
\usepackage{amssymb}
\usepackage{amsmath}
\usepackage{amsthm}
\usepackage{color}
\usepackage{mathrsfs}
\usepackage{txfonts}
\usepackage{enumerate}
\usepackage{indentfirst}
\usepackage{esint}

\def\loc{{\mathop\mathrm{\,loc\,}}}

\allowdisplaybreaks

\def\rr{{\mathbb{R}}}
\def\rn{{\mathbb{R}^n}}

\def\nn{{\mathbb N}}

\def\zp{\mathbb{Z}_{+}}
\def\xx{\mathcal{X}}
\def\q{Q_{0}}
\def\md{\mathcal{M}^{(\rm d)}_{Q}}
\def\BMO{\mathop\mathrm{\,BMO\,}}
\def\mdq{\mathcal{M}^{(\rm d)}_{\q}}
\def\jn{JN_{(p,q,s)_{\alpha}}(\mathcal{X})}
\def\jq{JN_{(p,q,s)_{\alpha}}(Q_{0})}

\def\jnsp{JN_{(p,q,s+1)_{\alpha}}(\mathcal{X})}
\def\jns{{JN_{(p,q,0)_{\alpha}}(\mathcal{X})}}
\def\njn{\widetilde{JN}_{(p,q,s)_{\alpha}}( \mathcal{X})}
\def\njnsp{\widetilde{JN}_{(p,q,s+1)_{\alpha}}( \mathcal{X})}
\def\njnq{\widetilde{JN}_{(p,1,s)_{\alpha}}( \mathcal{X})}
\def\njnqq{\widetilde{JN}_{(p,1,s)_{0}}( \q)}
\def\njns{\widetilde{JN}_{(p,q,0)_{\alpha}}( \mathcal{X})}
\def\njq{\widetilde{JN}_{(p,q,s)_{\alpha}}( \q)}
\def\njnaq{\widetilde{JN}_{(p,q,s)_{0}}( \q)}
\def\psq{\mathcal{P}_{s}(Q)}
\def\ps{\mathcal{P}_{s}(Q_{0})}
\def\c{C_{(s)}}
\def\psf{\mathcal{P}^{(s)}_{Q}(f)}
\def\lqq{L^{q}(Q_{0})}
\def\JN{John--Nirenberg--Campanato\ }
\def\qi{\{Q_{i}\}_{i}}
\def\fjn{\|f\|_{\jn}}

\def\fnjn{\|f\|_{\njn}}

\def\fnjq{\|f\|_{\njq}}
\def\RM{RM_{p,q,\az}(\xx)}

\def\Rm{RM_{p,1,\az}(\xx)}
\def\fRM{\|f\|_{\RM}}

\def\lql{L^{q}_{\rm{loc}}(\xx)}
\def\cx{\mathcal{C}_{\az,q,s}(\xx)}

\def\ncx{\widetilde{\mathcal{C}}_{\az,q,s}(\xx)}
\def\ncq{\widetilde{\mathcal{C}}_{\az,q,s}(\q)}
\def\fcx{\|f\|_{\cx}}

\def\fncx{\|f\|_{\ncx}}
\def\fncq{\|f\|_{\ncq}}
\def\qw{\widetilde{Q}}

\def\fz{\infty }
\def\az{\alpha}
\def\lf{\left}
\def\r{\right}
\def\wz{\widetilde}

\newtheorem{theorem}{Theorem}[section]
\newtheorem{lemma}[theorem]{Lemma}
\newtheorem{corollary}[theorem]{Corollary}
\newtheorem{proposition}[theorem]{Proposition}

\theoremstyle{definition}
\newtheorem{remark}[theorem]{Remark}
\newtheorem{definition}[theorem]{Definition}
\renewcommand{\appendix}{\par
\setcounter{section}{0}%
\setcounter{subsection}{0}%
\setcounter{subsubsection}{0}%
\gdef\thesection{\@Alph\c@section}%
\gdef\thesubsection{\@Alph\c@section.\@arabic\c@subsection}%
\gdef\theHsection{\@Alph\c@section.}%
\gdef\theHsubsection{\@Alph\c@section.\@arabic\c@subsection}%
\csname appendixmore\endcsname
}

\numberwithin{equation}{section}

\begin{document}
\arraycolsep=1pt

\title{\bf\Large New John--Nirenberg--Campanato-Type Spaces
Related to Both Maximal Functions and Their Commutators\footnotetext{\hspace{-0.35cm} 2020 {\it
Mathematics Subject Classification}. Primary 42B35;
Secondary 47B47, 46E30, 42B25.
\endgraf {\it Key words and phrases.}
Euclidean space, cube, John--Nirenberg--Campanato
space, Riesz--Morrey space, maximal function, commutator, John--Nirenberg inequality.
\endgraf This project is partially supported by the National
Key Research and Development Program of China
(Grant No.\ 2020YFA0712900)
and the National Natural Science Foundation of China
(Grant Nos.\ 11971058 and 12071197).
}}
\author{Pingxu Hu, Jin Tao and Dachun Yang\footnote{Corresponding author,
E-mail: \texttt{dcyang@bnu.edu.cn}/{\color{red} July 1, 2022}/Final version.}}
\date{}
\maketitle

\vspace{-0.9cm}

\begin{center}
\begin{minipage}{13cm}
{\small {\bf Abstract}\quad
Let $p,q\in [1,\infty]$, $\alpha\in{\mathbb{R}}$, and $s$ be a
non-negative integer. In this article, the authors introduce
a new function space
$\widetilde{JN}_{(p,q,s)_{\alpha}}(\mathcal{X})$
of John--Nirenberg--Campanato type,
where $\mathcal{X}$ denotes $\mathbb{R}^n$
or any cube $Q_{0}$ of $\mathbb{R}^n$
with finite edge length.
The authors give an equivalent characterization of
$\widetilde{JN}_{(p,q,s)_{\alpha}}(\mathcal{X})$ via both
the John--Nirenberg--Campanato space and the Riesz--Morrey space.
Moreover, for the particular case $s=0$,
this new space can be equivalently
characterized by both maximal functions and their commutators.
Additionally, the authors give some basic properties,
a good-$\lambda$ inequality, and a John--Nirenberg type inequality
for $\widetilde{JN}_{(p,q,s)_{\alpha}}(\mathcal{X})$.
}

\end{minipage}
\end{center}

\vspace{0.2cm}

\section{Introduction}

Throughout this article, a cube
$\q$ \emph{always means} that $\q$ has finite edge length, all its edges parallel
to the coordinate axes, but $\q$ is not necessary to be open or
closed. We \emph{always} let $\xx$ represent either $\rn$
or a cube $\q$ of $\rn$ with finite edge length.
For any integrable function $f$ and any cube $Q$, let
\[f_{Q}:=\fint_{Q}f:=\frac{1}{|Q|}\int_{Q}f.\]
Here and thereafter, in all integral representation,
in order to simplify the presentation, we \emph{always omit} the
differential $dx$ if there exists no confusion.

The John--Nirenberg space $JN_p$ has attracted
a lot of attention in recent years.
It is a byproduct appearing in the study of
John and Nirenberg \cite{JN61} on functions
with bounded mean oscillation
(the celebrated space BMO),
and was further used in the interpolation
theory by Stampacchia \cite{S65}.
Both $JN_p$ and BMO are function spaces
based on mean oscillations of functions, and we refer the reader to
\cite{be20,cds99,cl99,cs06,ins14,ins19,is17,jlx19,lmv20,lny18}
for more related researches.
Obviously, we have $JN_1(Q_0)=L^1(Q_0)$ and $JN_\fz(Q_0)=\BMO(Q_0)$;
see, for instance, \cite{tyyM}.
Moreover, an interesting result of Dafni et al. \cite{DHKY18}
shows the non-triviality of $JN_p(Q_0)$ with $p\in(1,\fz)$
via constructing a surprising function belonging to $JN_p(Q_0)$
but not an element of $L^p(Q_0)$.
This means that $JN_p(Q_0)$ is strictly larger than $L^p(Q_0)$.
Later, a $JN_p$-type space mixed with Campanato structure,
called the \JN space $\jn$, was introduced and studied
in \cite{tyyNA}; see Definition \ref{s2def2} below. Indeed, for
any $p,q\in[1,\fz]$, $s\in \zp$, and $\az \in\rr$,
$$
\jn=
\begin{cases}
\text{the John--Nirenberg space } JN_p(\xx), &q=1,\ s=0=\az,\\
\text{the Campanato space } \cx, & p=\fz,\\
\text{the space }\BMO(\xx),& p=\fz,\ \az=0,
\end{cases}
$$
and hence $\jn$ combines some futures of both
the John--Nirenberg space and the Campanato space.
However, as is mentioned in \cite{DHKY18},
the structure of $JN_p$ is largely a mystery,
and so does $\jn$.
For instance, even for some classical operators
(such as the Hardy--Littlewood maximal operator,
the Calder\'on-Zygmund operator, and the fractional integral),
it is still unclear whether or not it is
bounded on $JN_p(\xx)$ or $\jn$;
we refer the reader to
\cite{jtyyz21,jtyyz22a,jtyyz22b,km21,tyyM}
for some related studies.
The main purpose of this article is to investigate
the \JN space from the point of view of the negative part
$f^{-}:=-\min\{f,0\}$, and to shed some light on the structure
of John--Nirenberg-type spaces associated
with both maximal functions and their commutators.

The negative part reveals some significant properties of $\BMO(\rn)$
related to maximal functions.
Precisely, let $\mathcal{M}$ denote
the \emph{Hardy--Littlewood maximal operator}
defined by setting, for any $f\in L^1_{\rm loc}(\rn)$ (the set of all
locally integrable functions) and any $x\in\rn$,
$$\mathcal{M}(f)(x)
:=\sup_{\text{cube }Q\ni x}\fint_{Q} |f(y)|\,dy.$$
Moreover, for any given cube $Q$ of $\rn$ and for
any $f\in L^1_{\rm loc}(\rn)$ and $x\in\rn$, let
\begin{align}\label{MQ}
\mathcal{M}_Q(f)(x)
:=\sup_{\{\text{cube }Q_\ast:\ x\in
Q_\ast\subset Q\}}\fint_{Q_\ast} |f(y)|\,dy.
\end{align}
For any $b\in L^1_{\rm loc}(\rn)$, define the
\emph{commutator $[b,\mathcal{M}]$} by setting,
for any $f\in L^1_{\rm loc}(\rn)$ and $x\in\rn$,
\begin{align}\label{commutator}
[b,\mathcal{M}](f)(x):=b(x)\mathcal{M}(f)(x)-\mathcal{M}(bf)(x).
\end{align}
Then Bastero et al. \cite{BMR00} proved that
the following three statements are mutually equivalent:
\begin{enumerate}[\rm(i)]
\item for any (or some) $p\in(1,\fz)$,
the commutator $[b,\mathcal{M}]$ is bounded on $L^p(\rn)$;
\item $b\in\BMO(\rn)$ with $b^-\in L^\fz(\rn)$;
\item for any (or some) $q\in[1,\fz)$,
$$\sup\fint_Q \lf|b-\mathcal{M}_Q(b)\r|^q<\fz,$$
where the supremum is taken over all cubes $Q$ of $\rn$.
\end{enumerate}
Very recently, Wang and Shu \cite{WS22} further showed that,
if we replace $\mathcal{M}_Q(b)$ in (iii) by $|b|_Q$,
then the above equivalence still holds true.
Indeed, they obtain the equivalence between (ii) and
\begin{enumerate}
\item[\rm(iv)] ${\sup\fint_Q |b-|b|_Q|}<\fz$,
where the supremum is taken over all cubes $Q$ of $\rn$.
\end{enumerate}
This implies that (i) $\Longleftrightarrow$ (ii)
$\Longleftrightarrow$ (iii) $\Longleftrightarrow$ (iv).
Moreover, Wang and Shu \cite{WS22} also
studied this phenomenon
in the context of Campanato spaces,
and introduced a new function spaces of
Morrey--Campanato type in \cite{WS22}.
As applications, they obtained an integral
characterization of the non-negative
H\"older continuous functions.
Note that $f$ is non-negative if and only if $f^-=0$.
Thus, the structure of this new space of Morrey--Campanato type
is also completely determined by the classical Campanato space
and the boundedness  of the negative part $f^-$.

In this article, we introduce a new
John--Nirenberg--Campanato-type space $\njn$
with $p,q\in [1,\fz]$, $\az\in\rr$, and $s$
being a non-negative integer;
see Definition \ref{s2def2} below.
We give an equivalent characterization of $\njn$ via
the John--Nirenberg--Campanato space and  the
Riesz--Morrey space.
Moreover, for the particular case $s=0$,
this new space can be equivalently
characterized by maximal functions and their commutators.
Additionally, we give some basic properties,
a good-$\lambda$ inequality,
and a John--Nirenberg type inequality for $\njq$.

The only difference between the ``norms'' of $\jn$ and $\njn$ is that
we change the mean oscillation
$${\fint_{Q}|f-P^{(s)}_{Q}(f)|^{q}}\ \mathrm{into}\
{\fint_{Q}|f-P^{(s)}_{Q}(|f|)|^{q}},$$
where $P^{(s)}_{Q}(\cdot)$ denotes the unique polynomial
satisfying \eqref{s1eq1}.
Since $P^{(0)}_{Q}(f)$ coincides with $f_{Q}$,
this is a natural generalization of the above (iv).
Thus, a natural question appears:
\[\text{Whether or not $f^-$ is still a bridge
connecting $\jn$ and $\njn$?}\]
Indeed, we give a positive answer to this question
in Theorem \ref{s2thm1} below,
and the key ingredient is the Riesz--Morrey space $\RM$
studied in \cite{tyyBJMA,ZT2021}.
Correspondingly, the only difference between the
norms of $\jn$ and $\RM$ is that
we change the mean oscillation
$$\fint_{Q}\lf|f-P^{(s)}_{Q}(f)\r|^{q}\ \mathrm{into}\
\fint_{Q}|f|^{q};$$
see Definition \ref{s2def3} below.
Moreover, for the case $s=0$,
we further obtain a corresponding equivalence of above
(i) $\Longleftrightarrow$ (ii) $\Longleftrightarrow$
(iii) $\Longleftrightarrow$ (iv) on $\njn$ via maximal functions
and their commutators;
see Theorem \ref{s3thm1} and Corollary \ref{s3cor2} below.

The remainder of this article is organized as follows.

In Section \ref{s2},
we introduce a ``new'' \JN type space $\njn$,
and prove that
\begin{equation}\label{s1eq00}
\RM\subset \njn\subset \jn
\end{equation}
in Proposition \ref{s2prop1} below. Indeed,
Remark \ref{s2rmk4} shows that this is truly
a new space between the Riesz--Morrey space
and the John--Nirenberg space,
namely, the inclusions in \eqref{s1eq00} are proper.
Moreover, we give an equivalent characterization of $\njn$
in Theorem \ref{s2thm1} below
via the \JN space $\jn$ and the Riesz--Morrey space $\Rm$.
The proofs of these results rely heavily on \eqref{s1eq2} below
and the linearity of $P^{(s)}_{Q}(\cdot)$.

In Section \ref{s3}, we first give some basic properties
of $\njn$, including the monotonicity
on the subindices $p$, $q$, and $s$
(see Proposition \ref{s3prop2} below),
the limit behavior as $p\to\fz$ (see Proposition \ref{s3prop1}
and Corollary \ref{s3cor1} below),
and the invariance on the second subindex $q$
(see Proposition \ref{s3prop4} below).
Moreover, for the case $s=0$, $P^{(s)}_{Q}(|f|)=|f|_{Q}$
is closely connected with the maximal function $\mathcal{M}_Q(f)$,
and hence we further study the relation between
$\njn$ and maximal functions.
On one hand, the maximal operator is bounded on $L^q(\xx)$ with
$q\in (1,\fz)$
but not bounded on $L^1(\xx)$. On the other hand, when dominating
mean oscillations by maximal functions,
the linearity of $\|\cdot\|_{L^1(\xx)}$ plays a key role,
which is no longer feasible for $L^q(\xx)$ with $q\in (1,\fz)$;
see Remark \ref{rem-q<p} below.
To coordinate these two conflicts (namely, $q=1$ and $q>1$),
we give another characterization of $\njn$ in Proposition \ref{s3prop3} below.
With the aid of Propositions \ref{s3prop4} and \ref{s3prop3},
we flexibly change the index $q$ and hence
characterize $\njn$ via maximal functions in Theorem \ref{s3thm1} below.
Furthermore, a related characterization via their commutators
is also established in Corollary \ref{s3cor2} below.

In Section \ref{s4}, we prove a John--Nirenberg type inequality
and a good-$\lambda$ inequality for the space $\njq$,
respectively, in Theorem \ref{s4thm1} and Lemma \ref{s4lem2} below.
As an application, we use this John--Nirenberg type inequality
to obtain another proof of Proposition \ref{s3prop4} at
the end of this section.

Finally, we make some conventions on notation.
Throughout this article, $\xx$ always represents
either $\rn$ or a cube $\q$ of $\rn$ with finite edge length.
Let $\mathbb{N}:=\{1, 2, \dots\}$ and $\zp:=\mathbb{N}\cup \{0\}$.
For any real-valued function $f$, we use $f^{-}:=-\min\{f,0\}$
to denote its negative part.
Let ${\bf {1}}_{E}$ denote the characteristic function of
any set $E\subset \rn$, and $\psq$ the set of all polynomials of
degree not greater than $s\in\zp$ on $Q$.
We denote by both $C$ and $\wz{C}$ positive
constants which are independent of the main parameters, but
they may vary from line to line.
Moreover, we use
$C_{(\gamma,\ \beta,\ \dots)}$ to denote
a positive constant depending
on the indicated parameters $\gamma, \beta, \dots$. Constants
with subscripts, such as $C_{0}$ and $A_{1}$, do not change in
different occurrences. Moreover, the symbol $f\lesssim g$
represents that $f \leq Cg$ for some positive constant $C$.
If $f\lesssim g$ and $g\lesssim f$, we then write $f\thicksim g$.
If $f\leq Cg$ and
$g=h$ or $g\leq h$, we then write $f\lesssim g\thicksim h$
or $f\lesssim g\lesssim h$, rather than $f\lesssim g=h$ or
$f\lesssim g\leq h$. For any $p\in[1,\fz]$, let $p'$ be its
\emph{conjugate index}, that is, $p'$ satisfies $1/p+1/p'=1$.

\section{New Function Spaces of John--Nirenberg--Campanato
Type \label{s2}}

To introduce the space $\njn$, we first recall the following
basic concepts.
\begin{itemize}
\item For any $s\in \zp$ (the set of all
non-negative integers), $\psf$ denotes the unique polynomial
of degree not greater than $s$ such that, for any $|\gamma|\leq s$,
\begin{equation}\label{s1eq1}
\int_{Q}\left[f(x)-P^{(s)}_{Q}(f)(x)\right]x^{\gamma}\,dx=0,
\end{equation}
where $\gamma:=(\gamma_{1},\dots,\gamma_{n})\in\mathbb{Z}^{n}_{+}:
=(\mathbb{Z}_{+})^{n},\ |\gamma|:=\gamma_{1}+\cdots+\gamma_{n}$, and
$x^{\gamma}:=x^{\gamma_{1}}_{1}\cdots x^{\gamma_{n}}_{n}$ for any
$x:=(x_{1},\dots,x_{n})\in \rn$. It is well known that
$P^{(0)}_{Q}(f)=f_{Q}$
and, for any $s\in \zp$, there exists a constant $\c\in [1,\fz)$,
independent of both $f$ and $Q$, such that, for any $x\in Q$,
\begin{equation}\label{s1eq2}
\left| P_{Q}^{(s)}(f)(x)\right| \leq C_{(s)}\fint_{Q}|f|;
\end{equation}
see \cite[p.\,83]{TG80} and \cite[Lemma 4.1]{SLF95} for more details.

\item For any given $q\in [1,\fz)$, the symbol $L^{q}(\xx)$
denotes the
spaces of all the measurable functions $f$ on $\xx$ such that
\[\|f\|_{L^{q}(\xx)}:=\lf[\int_{\xx}|f(x)|^{q}\,dx\r]^{\frac{1}{q}}
<\fz\]
and the symbol $L^{q}_{\rm loc}(\xx)$ the set of all the measurable
functions $f$ on $\xx$ such that $f{\bf{1}}_{E}\in L^{q}(\xx)$ for any
bounded set $E\subset \xx$. Besides, when $q=\fz$,
$\|\cdot\|_{\xx}$ represents the essential supremum on $\xx$.
\item For any given $q\in [1,\fz]$ and for any measurable
function $f$, let
\begin{equation*}
\|f\|_{L^{q}(Q_{0},|Q_{0}|^{-1}dx)}
:=\lf[\fint_{Q_{0}}|f(x)|^{q}\,dx\r]^{\frac{1}{q}}.
\end{equation*}
\item For any given $q\in[1,\fz)$ and $s\in \zp$, the space
$L^{q}(Q_{0},|Q_{0}|^{-1}dx)/\ps$ is defined by setting
\begin{equation*}
L^{q}(Q_{0},|Q_{0}|^{-1}dx)/\ps:\ =\lf\{f\in \lqq:\
\|f\|_{L^{q}(Q_{0},|Q_{0}|^{-1}dx)/\ps}<\fz \r\},
\end{equation*}
where
\begin{equation*}
\|f\|_{L^{q}(Q_{0},|Q_{0}|^{-1}dx)/\ps}:=\inf_{m\in\ps}
\|f+m\|_{L^{q}(Q_{0},|Q_{0}|^{-1}dx)}.
\end{equation*}
\item The space $\BMO (\xx)$, introduced by John and
Nirenberg \cite{JN61}
in 1961 to study the functions of \emph{bounded mean oscillation},
is defined by setting
\[\BMO(\xx):=\lf\{f\in L^{1}_{\rm loc}(\xx):\ \|f\|_{\BMO(\xx)}=
\sup_{{\rm cube}\,Q\subset\xx}\fint_{Q}\lf|f-f_{Q}\r|<\fz\r\}
\]
with the supremum taken over all cubes $Q$ of $\xx$.
\end{itemize}

\begin{definition}\label{s2def1}
Let $\az \in\rr$, $q\in[1,\fz]$, and $s\in \zp$.
\begin{enumerate}[\rm(i)]
\item The \emph{Campanato space $\cx$},
introduced by Campanato \cite{C64},
is defined by setting
\[\cx:=\lf\{f\in\lql:\ \fcx<\fz\r\}\]
with
\[\fcx:=\sup\left\{|Q|^{-\alpha}
\left[\fint_{Q}\lf|f-P^{(s)}_{Q}(f)\r|^{q}\right]^{\frac{1}{q}}
\right\},\]
where $P^{(s)}_{Q}(f)$ is the same as in \eqref{s1eq1}
and the supremum is taken over all cubes $Q$ of $\xx$.

\item The \emph{space $\ncx$} is defined by setting
\[\ncx:=\lf\{f\in\lql:\ \fncx<\fz\r\}\]
with
\[\fncx:=\sup\left\{|Q|^{-\alpha}
\left[\fint_{Q}\lf|f-P^{(s)}_{Q}(|f|)\r|^{q}\right]^{\frac{1}{q}}
\right\},\]
where $P^{(s)}_{Q}(f)$ is the same as in \eqref{s1eq1}
and the supremum is taken over all cubes $Q$ of $\xx$.
\end{enumerate}
\end{definition}

\begin{remark}
\begin{enumerate}[(i)]
\item The ``norm'' $\|\cdot\|_{\cx}$ is defined modulo polynomials.
For simplicity, we regard $\cx$ as the quotient space
$\cx /\mathcal{P}_{s}(\xx)$. However, $\|\cdot\|_{\ncx}$ is
not a norm of $\lql$
because the triangular inequality does not hold true.
\item Comparing $\ncx$ in Definition \ref{s2def1}(ii) with
$\overline{\mathcal{L}}^{p,\lambda}(\xx)$ in \cite{WS22},
it is easy to find that
$$\widetilde{\mathcal{C}}_{\az,q,0}(\xx)
=\overline{\mathcal{L}}^{q,n(\az q+1)}(\xx),$$ where $n$ denotes
the dimension of $\xx$.
\end{enumerate}
\end{remark}

\begin{definition}\label{s2def2}
Let $q\in[1,\fz]$, $s\in \zp$, and $\az \in\rr$.
\begin{enumerate}[\rm(i)]
\item If $p\in [1,\fz)$, then the \emph{\JN space $\jn$}, introduced in \cite{tyyNA},
is defined by setting
\[
\jn:=\lf\{ f \in L^{q}_{\rm loc}(\xx) :\ \fjn <\fz \r\}
\]
with
\[
\fjn:=\sup\lf\{ \sum_{i}|Q_{i}|\lf[ |Q_{i}|^{-\alpha}
\lf\{\fint_{Q_{i}}\lf|f-P^{(s)}_{Q_{i}}
(f)\r|^{q}\r\}^{\frac{1}{q}}\r]^{p}\r\}^{\frac{1}{p}},
\]
where
$P^{(s)}_{Q_{i}}(f)$ for any $i$ is the same as in \eqref{s1eq1}
with $Q$ replaced by
$Q_{i}$ and the supremum is taken over all the collections of
interior pairwise disjoint cubes $\qi$ of $\xx$.
\item If $p\in [1,\fz)$, then the
\emph{John--Nirenberg--Campanato-type space $\njn$} is defined by setting
\[
\njn:=\lf\{ f \in L^{q}_{\rm loc}(\xx) :\ \fnjn <\fz \r\}
\]
with
\[
\fnjn:=\sup\lf\{ \sum_{i}|Q_{i}|\lf[ |Q_{i}|^{-\alpha}
\lf\{\fint_{Q_{i}}\lf|f-P^{(s)}_{Q_{i}}
(|f|)\r|^{q}\r\}^{\frac{1}{q}}\r]^{p}\r\}^{\frac{1}{p}},
\]
where $P^{(s)}_{Q_{i}}(|f|)$ for any $i$ is the same as in \eqref{s1eq1}
with both $Q$ replaced by
$Q_{i}$ and $f$ replaced by $|f|$,
and where the supremum is taken over all the collections of
interior pairwise disjoint cubes $\qi$ of $\xx$.
\item  Let
$JN_{(\fz,q,s)_{\az}}(\xx):=\cx$ and
$\wz{JN}_{(\fz,q,s)_{\az}}(\xx):=\ncx$.

\end{enumerate}
\end{definition}

\begin{remark}\label{s2rmk2}
Let $p$, $q$, $s$, and $\az$ be the same as in Definition \ref{s2def2}.
\begin{enumerate}[(i)]
\item Let $f$ be any non-negative function on $\xx$.
Then $f\in\jn$ if and only if $f\in\njn$.
\item $\|\cdot\|_{\njn}$ is not a norm of $\lql$
because the triangular inequality does not hold true.
\end{enumerate}
\end{remark}

The following Riesz--Morrey space was introduced in \cite{tyyBJMA}
as a bridge connecting Lebesgue spaces and Morrey spaces.

\begin{definition}\label{s2def3}
Let $p,q\in[1,\fz]$ and $\az\in\rr$.
Then the \emph{Riesz--Morrey space $\RM$} is defined by setting
\[
\RM:=\lf\{ f \in \lql :\ \fRM <\fz \r\}
\]
with
\begin{equation*}
\fRM:=
\begin{cases}
{\displaystyle
\sup\left\{ \sum_{i}|Q_{i}|\left[ |Q_{i}|^{-\az}
\left\{\fint_{Q_{i}}|f|^{q}\right\}
^{\frac{1}{q}}\right]^{p}\right\}^{\frac{1}{p}}}
&{\rm if}\ p\in[1,\fz),\ q\in[1,\fz],\\
{\displaystyle
\sup_{{\rm cube}\ Q\subset \xx}|Q|^{-\az}\lf[\fint_{Q}|f|^{q}\r]
^{\frac{1}{q}}}
&{\rm if}\ p=\fz,\ q\in[1,\fz],
\end{cases}
\end{equation*}
where the first supremum is taken over all the collections of
interior pairwise disjoint cubes $\qi$ of $\xx$ and the second
supremum is taken over all cubes $Q$ of $\xx$.
\end{definition}

\begin{remark}\label{s2rmk3}
The relation between Riesz--Morrey spaces and Lebesgue spaces
is completely clarified over all indices in
\cite[Corollary 3.7]{ZT2021}. For the convenience
of the reader, we list
all cases of $\Rm$ as follows.
\begin{enumerate}[(i)]
\item Let $p\in(1,\fz]$. Then
\[
RM_{p,1,\az}(\rn)
\begin{cases}
=L^{1}(\rn) &{\rm if}\ \az=\frac{1}{p}-1,\\
\supsetneqq L^{\frac{p}{1-p\az}}(\rn) &{\rm if}\ \az
\in \lf(\frac{1}{p}-1,0\r),\\
=L^{p}(\rn)  &{\rm if}\ \az=0,\\
=\{0\} &{\rm if}\ \az\in \lf(-\fz,\frac{1}{p}-1\r)\cup (0,\fz).
\end{cases}
\]
In particular, for any $\az\in(-1,0)$,
$RM_{\fz,1,\az}(\rn)=M_{1,\az}(\rn)$ which is the Morrey space.
\item
\[
RM_{1,1,\az}(\rn)=
\begin{cases}
L^{1}(\rn) &{\rm if }\ \az=0,\\
\{0\} &{\rm if }\ \az\in\rr \setminus\{0\}.
\end{cases}
\]
\item Let $p\in(1,\fz]$ and $Q_{0}$ be any cube of
$\rn$. Then
\[
RM_{p,1,\az}(Q_{0})
\begin{cases}
=L^{1}(Q_{0}) &{\rm if}\  \az\in\lf( -\fz,\frac{1}{p}-1\r],\\
\supsetneqq L^{\frac{p}{1-p\az}}(Q_{0}) &{\rm if}\ \az
\in \lf(\frac{1}{p}-1,0\r),\\
=L^{p}(Q_{0})  &{\rm if}\  \az=0,\\
=\{0\} &{\rm if}\  \az\in(0,\fz).
\end{cases}
\]
In particular, $RM_{\fz,1,\az}(Q_{0})=M_{1,\az}(Q_{0})$
if $\az\in(-1,0)$.
\item Let $Q_{0}$ be any cube of
$\rn$. Then
\[
RM_{1,1,\az}(Q_{0})=
\begin{cases}
L^{1}(Q_{0}) &{\rm if}\  \az\in( -\fz,0],\\
\{0\}  &{\rm if}\  \az\in (0,\fz).
\end{cases}
\]
\end{enumerate}
\end{remark}

\begin{proposition}\label{s2prop1}
Let $p,q\in[1,\infty]$, $s\in\zp$, and $\alpha \in\rr$. Then
\begin{equation*}
\RM \subset \njn\subset \jn
\end{equation*}
and
\begin{equation}\label{s2prop1eq1}
\lf[1+\c\r]^{-1}\|\cdot\|_{\jn}\leq\|\cdot\|_{\njn}
\leq\lf[1+\c\r]\|\cdot\|_{\RM},
\end{equation}
where $\c\in[1,\fz)$ is the same as in \eqref{s1eq2}.
\end{proposition}

\begin{proof}
To show Proposition \ref{s2prop1}, it is enough
to prove that \eqref{s2prop1eq1} holds true.

First, we show the first inequality of \eqref{s2prop1eq1}.
To this end, for
any $f\in \lql$ and any interior pairwise disjoint
cubes $\qi$ of $\xx$, from the Minkowski inequality,
the linearity of $\{P^{(s)}_{Q_{i}}\}_{i}$,
\eqref{s1eq2},
and the H\"{o}lder inequality, it follows that
\begin{align*}
&\left\{\sum_{i}|Q_{i}|\left[ |Q_{i}|
^{-\alpha}\left\{\fint_{Q_{i}}\lf|f-P^{(s)}_{Q_{i}}(f)\r|^{q}
\right\}^{\frac{1}{q}}\right]^{p}\right\}^{\frac{1}{p}}\\
&\quad\leq
\left\{\sum_{i}|Q_{i}|^{1-p\alpha}\left[
\fint_{Q_{i}}\lf|f-P^{(s)}_{Q_{i}}(|f|)\r|^{q}\right]
^{\frac{p}{q}}\right\}^{\frac{1}{p}}\\
&\quad\quad+\left\{\sum_{i}|Q_{i}|^{1-p\alpha}\left[
\fint_{Q_{i}}\lf|P^{(s)}_{Q_{i}}(|f|)-P^{(s)}_{Q_{i}}(f)\r|
^{q}\right]
^{\frac{p}{q}}\right\}^{\frac{1}{p}}\\
&\quad\le \fnjn
+\left\{\sum_{i}|Q_{i}|^{1-p\alpha}\left[
\fint_{Q_{i}}\lf|P^{(s)}_{Q_{i}}\lf(P^{(s)}_{Q_{i}}(|f|)-f\r)\r|
^{q}\right]
^{\frac{p}{q}}\right\}^{\frac{1}{p}}\\
&\quad\leq\fnjn
+\c\left\{\sum_{i}|Q_{i}|^{1-p\alpha}\left[
\fint_{Q_{i}}\lf|P^{(s)}_{Q_{i}}(|f|)-f\r|
\right]
^{p}\right\}^{\frac{1}{p}}\\
&\quad\leq
\fnjn
+\c\left\{\sum_{i}|Q_{i}|^{1-p\alpha}\left[
\fint_{Q_{i}}\lf|P^{(s)}_{Q_{i}}(|f|)-f\r|^{q}
\right]
^{\frac{p}{q}}\right\}^{\frac{1}{p}}\\
&\quad\leq \lf[1+\c\r]\fnjn,
\end{align*}
which implies that
the first inequality of \eqref{s2prop1eq1} holds true.

Next, we show the second inequality of \eqref{s2prop1eq1}.
To this end, for any $f\in \lql$ and any interior pairwise disjoint
cubes $\qi$ of $\xx$, by the Minkowski inequality, \eqref{s1eq2},
and the H\"{o}lder inequality, we conclude that
\begin{align*}
&\left\{\sum_{i}|Q_{i}|\left[ |Q_{i}|
^{-\alpha}\left\{\fint_{Q_{i}}\lf|f-P^{(s)}_{Q_{i}}(|f|)\r|^{q}
\right\}^{\frac{1}{q}}\right]^{p}\right\}^{\frac{1}{p}}\\
&\quad\leq
\left\{\sum_{i}|Q_{i}|^{1-p\alpha}\left[
\fint_{Q_{i}}\lf|f\r|^{q}\right]
^{\frac{p}{q}}\right\}^{\frac{1}{p}}
+\left\{\sum_{i}|Q_{i}|^{1-p\alpha}\left[
\fint_{Q_{i}}\lf|P^{(s)}_{Q_{i}}(|f|)\r|^{q}\right]
^{\frac{p}{q}}\right\}^{\frac{1}{p}}\\
&\quad\leq
\fRM+\c \left\{\sum_{i}|Q_{i}|^{1-p\alpha}\left[
\fint_{Q_{i}}\lf|f\r|\right]
^{p}\right\}^{\frac{1}{p}}\\
&\quad\leq\fRM+\c \left\{\sum_{i}|Q_{i}|^{1-p\alpha}\left[
\fint_{Q_{i}}\lf|f\r|^{q}\right]
^{\frac{p}{q}}\right\}^{\frac{1}{p}}\\
&\quad
\leq \lf[1+\c\r]\fRM,
\end{align*}
which implies that
the second inequality of  \eqref{s2prop1eq1} holds true.
Therefore, \eqref{s2prop1eq1} holds true,
which completes the proof of Proposition \ref{s2prop1}.
\end{proof}

\begin{remark}\label{s2rmk4}
Let $p\in(1,\fz)$ and $Q$ denote any given cube of $\rn$.
Then we claim that
$$RM_{p,1,0}(Q)\subsetneqq \wz{JN}_{(p,1,0)_{0}}(Q)
\subsetneqq JN_{(p,1,0)_{0}}(Q).$$
Indeed, by \cite[Proposition 1]{tyyBJMA}, we have
$RM_{p,1,0}(Q)=L^{p}(Q)$.
Moreover, it was proved in \cite[Corollary 4.2]{DHKY18} that there
exists a non-negative function $f\in JN_{(p,1,0)_{0}}(Q)\setminus
L^{p}(Q)$. From
Remark \ref{s2rmk2}(i), we infer that this
non-negative function $f\in\wz{JN}_{(p,1,0)_{0}}(Q)$.
Therefore, we have
\[ f\in\wz{JN}_{(p,1,0)_{0}}(Q)
\setminus L^{p}(Q)=\wz{JN}_{(p,1,0)_{0}}(Q)
\setminus RM_{p,1,0}(Q).\]
Moreover, let $g:=-f\leq 0$. Then we have
$g\in JN_{(p,1,0)_{0}}(Q)$ and
$g^{-}=f\notin L^{p}(Q)=RM_{p,1,0}(Q)$. Thus, from
Theorem \ref{s2thm1} below,
we deduce that
\[g\in JN_{(p,1,0)_{0}}(Q)\setminus \wz{JN}_{(p,1,0)_{0}}(Q).\]
Therefore, the above claim holds true.
\end{remark}

Now, we state the first main result of this article.

\begin{theorem}\label{s2thm1}
Let $p,q\in[1,\infty]$, $s\in\zp$,
and $\alpha \in\rr$. Then $f\in\njn$ if and only if
$f\in\jn$ and $f^{-}\in\Rm$. Moreover, for any
$f\in L^q_\loc(\xx)$,
$$\|f\|_{\njn}\sim \lf[\|f\|_{\jn}+\|f^{-}\|_{\Rm}\r]$$
with the positive equivalence constants
depending only on both $s$ and $n$.
\end{theorem}

\begin{proof}
Let $p$, $q$, $s$, and $\az$ be the same as in the present theorem.
We only consider the case $p<\fz$ because the case $p=\fz $ can
be similarly proved.

We first claim that $\|(\cdot)^{-}\|_{\Rm}
\lesssim \|\cdot\|_{\njn}$.
Indeed, for any cube $Q\subset\xx$ and any $f\in \lql$,
from the linearity of $P^{(s)}_{Q}$, $P^{(s)}_{Q}
(P^{(s)}_{Q}(|f|))=P^{(s)}_{Q}(|f|)$, and \eqref{s1eq2}, we deduce that
\begin{align}\label{s2thm1eq1}
\lf|2P^{(s)}_{Q}(f^{-})\r|&=\lf|P^{(s)}_{Q}(|f|-f)\r|
=\lf|P^{(s)}_{Q}(|f|)-P^{(s)}_{Q}(f)\r|=\lf|P^{(s)}_{Q}
\lf(P^{(s)}_{Q}(|f|)\r)-P^{(s)}_{Q}(f)\r|\\
&=\lf|P^{(s)}_{Q}\lf(P^{(s)}_{Q}(|f|)-f\r)\r|\leq
\c\fint_{Q}\lf|P^{(s)}_{Q}(|f|)-f\r|.\notag
\end{align}
Moreover, by \eqref{s1eq1}, we obtain
\[\fint_{Q}|f^{-}|=\fint_{Q}f^{-}=\fint_{Q}P^{(s)}_{Q}(f^{-})\leq
\fint_{Q}\lf|P^{(s)}_{Q}(f^{-})\r|\]
Therefore, for any $f\in\njn$ and any interior pairwise disjoint
cubes $\qi$ of $\xx$, from this, \eqref{s2thm1eq1},
and the H\"{o}lder inequality, it follows that
\begin{align*}
&\left\{\sum_{i}|Q_{i}|^{1-p\alpha}\left[ \fint_{Q_{i}}|f^{-}|
\right]^{p}\right\}^{\frac{1}{p}}\\
&\quad\leq\left\{\sum_{i}|Q_{i}|^{1-p\alpha}
\left[ \left\{\fint_{Q_{i}}
\lf|P^{(s)}_{Q_{i}}(f^{-})\r|\right\}
\right]^{p}\right\}^{\frac{1}{p}}\\
&\quad\leq\frac{\c}{2}\left\{\sum_{i}|Q_{i}|^{1-p\alpha}\left[
\fint_{Q_{i}}\lf|f-P^{(s)}_{Q_{i}}(|f|)\r|
\right]^{p}\right\}^{\frac{1}{p}}\\
&\quad\leq\frac{\c}{2}\left\{\sum_{i}|Q_{i}|^{1-p\alpha}\left[
\left\{\fint_{Q_{i}}\lf|f-P^{(s)}_{Q_{i}}(|f|)\r|^{q}\right\}
^{\frac{1}{q}}\right]^{p}\right\}^{\frac{1}{p}}\\
&\quad\leq\frac{\c}{2}\fnjn<\fz,
\end{align*}
which implies that
\begin{equation}\label{s2thm1eq02}
    \|f^{-}\|_{\Rm}\leq\frac{\c}{2}\fnjn<\fz
\end{equation}
and hence $f^{-}\in\Rm$. This shows that the above claim
holds true.

Moreover, using Proposition \ref{s2prop1}, we find that $f\in\jn$
and $$\|\cdot\|_{\jn}\lesssim\|\cdot\|_{\njn}.$$
To sum up,
$ \lf[\|\cdot\|_{\jn}+\|(\cdot)^{-}\|_{\Rm}\r]\lesssim
\|\cdot\|_{\njn}$ holds true.

It remains to prove
$ \|\cdot\|_{\njn}\lesssim
\lf[\|\cdot\|_{\jn}+\|(\cdot)^{-}\|_{\Rm}\r]$.
For any $f\in \jn$ with $f^{-}\in\Rm$,
and any interior pairwise
disjoint cubes $\qi$ of $\xx$,
from the Minkowski inequality, the linearity of $P^{(s)}_{Q_{i}}$, and \eqref{s1eq2},
we deduce that
\begin{align*}
&\left\{\sum_{i}|Q_{i}|\left[ |Q_{i}|^{-\alpha}\left\{\fint_{Q_{i}}
\lf|f-P^{(s)}_{Q_{i}}(|f|)\r|^{q}
\right\}^{\frac{1}{q}}\right]^{p}\right\}^{\frac{1}{p}}\\
&\quad\leq \left\{\sum_{i}|Q_{i}|^{1-p\alpha}
\left[ \left\{\fint_{Q_{i}}
\lf|f-P^{(s)}_{Q_{i}}(f)\r|^{q}\right\}
^{\frac{1}{q}}+\left\{\fint_{Q_{i}}\lf|P^{(s)}_{Q_{i}}(f)-P^{(s)}
_{Q_{i}}(|f|)\r|^{q}\right\}^{\frac{1}{q}}\right]^{p}\right\}
^{\frac{1}{p}}\\
&\quad=\left\{ \sum_{i}|Q_{i}|^{1-p\alpha}\left[
\left\{\fint_{Q_{i}}\lf|f-P^{(s)}_{Q_{i}}(f)\r|^{q}\right\}
^{\frac{1}{q}}+2\left\{\fint_{Q_{i}}\lf|P^{(s)}_{Q_{i}}(f^{-})\r|
^{q}\right\}^{\frac{1}{q}}\right]^{p}\right\}
^{\frac{1}{p}}\\
&\quad\leq\left[\sum_{i}|Q_{i}|^{1-p\alpha}\left\{
\fint_{Q_{i}}\lf|f-P^{(s)}_{Q_{i}}(f)\r|^{q}\right\}
^{\frac{p}{q}}\right]^{\frac{1}{p}}\\
&\quad \quad+2\left[\sum_{i}|Q_{i}|
^{1-p\alpha}\left\{\fint_{Q_{i}}\lf|P^{(s)}_{Q_{i}}(f^{-})\r|
^{q}\right\}^{\frac{p}{q}}\right]^{\frac{1}{p}}\\
&\quad\leq\fjn+2\c\left[\sum_{i}|Q_{i}|^{1-p\alpha}\left\{
\fint_{Q_{i}}|f^{-}|\right\}^{p}
\right]^{\frac{1}{p}}\\
&\quad\leq\fjn+2\c\|f^{-}\|_{\Rm}<\fz,
\end{align*}
which further implies that
\begin{equation}\label{s2thm1eq03}
    \fnjn\leq\fjn+2\c\|f^{-}\|_{\Rm}<\fz
\end{equation}
and hence $f\in \njn$. This shows that
$$ \|\cdot\|_{\njn}\lesssim
\lf[\|\cdot\|_{\jn}+\|(\cdot)^{-}\|_{\Rm}\r],$$
which completes the proof of Theorem \ref{s2thm1}.
\end{proof}

\begin{remark}\label{s2rmk5}
We have the following observations on Theorem \ref{s2thm1}.
\begin{enumerate}[(i)]
\item This theorem gives an equivalent characterization of  $\njn$.
\item Let $p\in(1,\fz)$, $q\in[p,\fz)$, $\az=0$, $s\in\zp$,
and $Q_{0}$ be a
cube of $\rn$. By Remark \ref{s2rmk3}(iii), we have
$$RM_{p,1,0}(Q_{0})=L^{p}(Q_{0})\supset L^{q}(Q_{0})
=L^{q}_{\rm loc}(Q_{0})\supset JN_{(p,q,s)_{0}}(Q_{0}).$$
For any $f\in JN_{(p,q,s)_{0}}(Q_{0})$, from this,
we deduce that $f^{-}\in RM_{p,1,0}(Q_{0})$. By this
and Theorem \ref{s2thm1}, we have
$$\wz{JN}_{(p,q,s)_{0}}(Q_{0})=JN_{(p,q,s)_{0}}(Q_{0}).$$
From this and \cite[Proposition 2.5]{tyyNA}, it follows that
\[|Q_{0}|^{-\frac{1}{p}}\njq=|Q_{0}|^{-\frac{1}{p}}\jq
=L^{q}(Q_{0},|Q_{0}|^{-1}dx)
/\mathcal{P}_{s}(Q_{0}).\]
\item Let $p,q\in[1,\fz)$, $s\in \zp$,
$\az\in(\frac{s+1}{n},\fz)$, and $Q_{0}$ be a cube of $\rn$.
Using both \cite[Proposition 2.9]{tyyNA} and
\cite[Lemma 3.1]{BBP15}, we have $\jq= \mathcal{P}_{s}(Q_{0})$.
Moreover,
by Remark \ref{s2rmk3}, we find that
$RM_{p,1,\az}(Q_{0})=\{0\}$.
Therefore, combining these and Theorem \ref{s2thm1},
we obtain
$$\wz{JN}_{(p,q,s)_{\az}}(Q_{0})
=\lf\{f\in\mathcal{P}_{s}(Q_{0}):f\geq 0\r\}.$$
\item Observe that $\mathcal{C}_{0,1,0}(\xx)=\BMO(\xx)$
and also that $RM_{\fz,1,0}=L^{\fz}(\xx)$ due to Remark \ref{s2rmk3}.
From these and Theorem \ref{s2thm1}, we deduce that
$f\in \wz{\mathcal{C}}_{0,1,0}(\xx)$ if and only if
$f\in \BMO(\xx)$ and $f^{-}\in L^{\fz}(\xx)$.
This coincides with \cite[Theorem 3.5]{WS22}.
\end{enumerate}
\end{remark}

\section{Basic Properties and Characterizations
via Maximal Functions\label{s3}}

In this section, we first give some basic properties of the spaces $\njn$
and then characterize them when $s=0$ via both maximal
functions and their commutators.

\begin{proposition}\label{s3prop2}
Let $p,q\in[1,\infty]$, $s\in \zp$, and $\alpha \in\rr$.
Then
\begin{enumerate}[\rm(i)]
\item for any $\widetilde{p}\in[p,\fz]$,
$\widetilde{JN}_{(\widetilde{p},q,s)_{\alpha}}(Q_0)
\subset
\widetilde{JN}_{(p,q,s)_{\alpha}}(Q_0)$ and
$$|Q_0|^{-1/p}
\|\cdot\|_{\widetilde{JN}_{(p,q,s)_{\alpha}}(Q_0)}
\le
|Q_0|^{-1/\widetilde{p}}
\|\cdot\|_{\widetilde{JN}_{(\widetilde{p},q,s)_{\alpha}}(Q_0)};$$

\item for any $\widetilde{q}\in[q,\fz]$,
$\widetilde{JN}_{(p,\widetilde{q},s)_{\alpha}}(\xx)
\subset
\widetilde{JN}_{(p,q,s)_{\alpha}}(\xx)$ and
$$\|\cdot\|_{\widetilde{JN}_{(p,q,s)_{\alpha}}(\xx)}
\le
\|\cdot\|_{\widetilde{JN}_{(p,\widetilde{q},s)_{\alpha}}(\xx)};$$

\item $\njnsp\subset\njn$ and
\begin{equation}\label{s3prop2eq00}
\|\cdot\|_{\njnsp}\leq\lf\{\lf[1+C_{(s+1)}\r]\lf[1+\c\r]+\c C_{(s+1)}
\r\}\|\cdot\|_{\njn},
\end{equation}
where $\c\in[1,\fz)$ is the same as in \eqref{s1eq2}.
\end{enumerate}
\end{proposition}

\begin{proof}
First, (i) follows from the Jensen inequality
and (ii) follows from the H\"older inequality;
we omit the details here.

It remains to prove (iii) of the present proposition.
To this end, it suffices to show that
\eqref{s3prop2eq00} holds true.
We only consider the case $p<\fz$ because the case $p=\fz $ can
be similarly proved.

Let $p\in[1,\infty)$, $q\in[1,\infty]$,
and $\alpha \in\rr$.
For any $s\in \zp$, we claim that
\begin{equation}\label{s3prop1eq00}
\|\cdot\|_{\jnsp}\leq\lf[1+C_{(s+1)}\r]\|\cdot\|_{\jn}.
\end{equation}
Indeed, for any $f\in \lql$ and any interior pairwise
disjoint cubes $\qi$
of $\xx$, by the Minkowski inequality,
$P^{(s)}_{Q_{i}}(f)=P^{(s+1)}_{Q_{i}}P^{(s)}_{Q_{i}}(f)$, the linearity
of $P^{(s+1)}_{Q_{i}}$, and \eqref{s1eq2}, we have
\begin{align*}
&\left\{\sum_{i}|Q_{i}|\left[ |Q_{i}|^{-\alpha}\left\{\fint_{Q_{i}}
\lf|f-P^{(s+1)}_{Q_{i}}(f)\r|^{q}
\right\}^{\frac{1}{q}}\right]^{p}\right\}^{\frac{1}{p}}\\
&\quad\leq \left\{\sum_{i}|Q_{i}|^{1-p\alpha}
\left[ \left\{\fint_{Q_{i}}
\lf|f-P^{(s)}_{Q_{i}}(f)\r|^{q}\right\}
^{\frac{1}{q}}\right]^{p}\right\}
^{\frac{1}{p}}\\
&\qquad+\left\{\sum_{i}|Q_{i}|^{1-p\alpha}
\left[\left\{\fint_{Q_{i}}\lf|P^{(s)}_{Q_{i}}(f)
-P^{(s+1)}
_{Q_{i}}(f)\r|^{q}\right\}^{\frac{1}{q}}\right]^{p}\right\}
^{\frac{1}{p}}\\
&\quad\le\fjn+\left\{\sum_{i}|Q_{i}|^{1-p\alpha}\left[\left\{\fint_{Q_{i}}\lf|P^{(s+1)}
_{Q_{i}}\lf(f-P^{(s)}_{Q_{i}}(f)
\r)\r|^{q}\right\}^{\frac{1}{q}}\right]^{p}\right\}
^{\frac{1}{p}}\\
&\quad\leq\fjn+\left\{\sum_{i}|Q_{i}|^{1-p\alpha}
\left[C_{(s+1)}\left\{\fint_{Q_{i}}
\lf|f-P^{(s)}_{Q_{i}}(f)\r|\right\}\right]^{p}\right\}
^{\frac{1}{p}}\\
&\quad\leq\fjn+C_{(s+1)}\left\{\sum_{i}|Q_{i}|^{1-p\alpha}
\left[ \left\{\fint_{Q_{i}}
\lf|f-P^{(s)}_{Q_{i}}(f)\r|^{q}\right\}
^{\frac{1}{q}}\right]^{p}\right\}
^{\frac{1}{p}}\\
&\quad\leq\lf[1+C_{(s+1)}\r] \fjn,
\end{align*}
which implies that the above claim \eqref{s3prop1eq00} holds true.

Now, for any $s\in \zp$ and $f\in\lql$, from \eqref{s2thm1eq03}, the above claim,
Proposition \ref{s2prop1},
and \eqref{s2thm1eq02}, we deduce that
\begin{align*}
\|f\|_{\njnsp}&\leq\|f\|_{\jnsp}+2C_{(s+1)}\|f^{-}\|_{\Rm}\\
&\leq\lf[1+C_{(s+1)}\r]\fjn+2C_{(s+1)}\|f^{-}\|_{\Rm}\\
&\leq\lf\{\lf[1+C_{(s+1)}\r]\lf[1+\c\r]+\c C_{(s+1)}\r\}\|f\|_{\njn},
\end{align*}
which implies that \eqref{s3prop2eq00} holds true.
This finishes the proof of Proposition \ref{s3prop2}.
\end{proof}

\begin{proposition}\label{s3prop1}
Let $\az\in\rr$, $q\in[1,\fz)$, and $s\in\zp$. Then, for any
$$f \in \bigcup_{r\in [1,\infty)}\bigcap_{p\in [r,\infty)}\njn,$$
it holds true that
\[\lim_{p\rightarrow\fz}\fnjn=\fncx.\]
\end{proposition}

\begin{proof}
Let $\az$, $q$, and $s$ be the same as in the present proposition.
Let $\qw$ be a cube of $\xx$ and $\qi$ a collection of
interior pairwise disjoint cubes of $\xx$ which contains
$\qw$ as its element. Then, for any $p\in[1,\fz)$ and
$f\in L^{q}_{\rm loc}(\xx)$,
\begin{align*}
\fnjn
&\geq\left\{\sum_{i}|Q_{i}|\left[|Q_{i}|^{-\alpha}\lf\{\fint_{Q_{i}}
\lf|f-P^{(s)}_{Q_{i}}(|f|)\r|^{q}\r\}^{\frac{1}{q}}
\right]^{p}\right\}
^{\frac{1}{p}} \\
&\geq \lf|\widetilde{Q}\r|^{\frac{1}{p}}
\lf|\widetilde{Q}\r|^{-\alpha}\left[
\fint_{\widetilde{Q}}\lf|f-P^{(s)}_{\widetilde{Q}}(|f|)\r|^{q}\right]
^{\frac{1}{q}}.
\end{align*}
Thus,
\[
\liminf_{p\rightarrow
\infty}\|f\|_{\njn}\geq
\lf|\widetilde{Q}\r|^{-\alpha}\left[
\fint_{\widetilde{Q}}\lf|f-P^{(s)}_{\widetilde{Q}}(|f|)\r|^{q}\right]
^{\frac{1}{q}}.
\]
By the arbitrariness of $\qw$, we obtain
$\liminf_{p\rightarrow \infty}\fnjn\geq \fncx$ for any
$f\in L^{q}_{\rm loc}(\xx)$.

Now, let $f \in \bigcup_{r\in [1,\infty)}\bigcap_{p\in
[r,\infty)}\njn$. Then there exists some $r_{0}\in[1,\fz)$
such that $f\in\njn$ for any $p\in[r_{0},\fz)$. We now show that
\begin{equation}\label{s3prop1eq1}
\limsup_{p\rightarrow\fz}\fnjn\leq\fncx.
\end{equation}
Indeed, if $\fncx=\fz$, then \eqref{s3prop1eq1} holds true trivially.
If $\fncx$ is finite, then, without loss of generality, we may assume
$\fncx=1$ because $\|\cdot\|_{\njn}$ and
$\|\cdot\|_{\ncx}$ are positively homogeneous.
Then, for any $p\in[r_{0},\fz)$, we have
\begin{align*}
\|f\|^{p}_{\njn} &=\sup \sum_{i}|Q_{i}|\left\{
|Q_{i}|^{-\alpha}\lf[\fint_{Q_{i}}\lf|f-P^{(s)}_{Q_{i}}(|f|)\r|
^{q}\r]
^{\frac{1}{q}}\right\}^{p}  \\
&\leq \sup \sum_{i}|Q_{i}|\left\{
|Q_{i}|^{-\alpha}\lf[\fint_{Q_{i}}\lf|f-P^{(s)}_{Q_{i}}(|f|)\r|
^{q}\r]
^{\frac{1}{q}}\right\}^{r_{0}}
\leq \|f\|^{r_{0}}_{\njn},
\end{align*}
where the suprema are taken over all the collections of interior
pairwise disjoint cubes $\qi$ of $\xx$ and the first inequality
holds true because
$$|Q_{i}|^{-\alpha}\lf\{\fint_{Q_{i}}
\lf|f-P^{(s)}_{Q_{i}}(|f|)\r|^{q}\r\}
^{\frac{1}{q}}\leq\fncx=1.$$
Letting $p\rightarrow\fz$, we then obtain
\[\limsup_{p\rightarrow\fz}\fnjn\leq1=\fncx,\]
and hence \eqref{s3prop1eq1} holds true.
This finishes the proof of Proposition \ref{s3prop1}.
\end{proof}

\begin{remark}\label{s3rmk1}
Let $\az$, $q$, and $s$ be the same as in Proposition
\ref{s3prop1}. From Proposition
\ref{s3prop1}, it follows that,
if \[f \in \bigcup_{r\in [1,\infty)}\bigcap_{p\in [r,\infty)}
\njn \quad
\text{and} \quad \liminf_{p\rightarrow \infty}\fnjn<\fz,\] then
\[f\in\ncx \quad\text{and} \quad \fncx=\lim_{p\rightarrow\fz}\fnjn.\]
\end{remark}

\begin{corollary}\label{s3cor1}
Let $q\in[1,\fz)$, $\az\in\rr$, $s\in\zp$, and $\q$ be a cube of
$\rn$. Then
\[\ncq=\lf\{f\in\bigcap_{p\in[1,\fz)}\njq:\ \lim_{p\rightarrow\fz}
\fnjq<\fz\r\}\]
and, for any $f\in\ncq$,
\[\fncq=\lim_{p\rightarrow\fz}\fnjq.\]
\end{corollary}

\begin{proof}
Let $q$, $\az$, and $s$ be the same as in the present corollary.
For any given $$f\in\bigcap_{p\in[1,\fz)}\njq\subset\bigcup_{r\in
[1,\infty)}\bigcap_{p\in [r,\infty)}\njq $$ with
$\lim_{p\rightarrow\fz}\fnjq<\fz$, by Remark \ref{s3rmk1},
we find that $f\in\ncq$ and $\|f\|_{\ncq}=\lim_{p\rightarrow\fz}\fnjq$.

Conversely, for any given $f\in\ncq=\wz{JN}_{(\fz,q,s)_{\az}}(Q_0)$,
from Proposition \ref{s3prop2}(i), it follows that
$f\in\njq$ for any $p\in[1,\fz)$,
which implies that
$$f\in\bigcap_{p\in[1,\fz)}\njq\subset\bigcup_{r\in
[1,\infty)}\bigcap_{p\in [r,\infty)}\njq .$$
Therefore, by Proposition \ref{s3prop1}, we obtain
\[\lim_{p\rightarrow\fz}\fnjq=\fncq<\fz.\]
This finishes the proof of Corollary \ref{s3cor1}.
\end{proof}

\begin{proposition}\label{s3prop4}
Let $1\leq q<p\leq \fz$, $s\in\zp$, and $\az\in[0,\fz)$.
Then
\[\njn=\wz{JN}_{(p,1,s)_{\az}}(\xx) \quad\text{and} \quad
\|\cdot\|_{\njn}\sim\|\cdot\|_{\wz{JN}_{(p,1,s)_{\az}}(\xx)}.\]
\end{proposition}

\begin{proof}
Let $p$, $q$, $s$, and $\az$ be the same as the present
proposition. We only consider the case $p<\fz$
because the case $p=\fz $ can be similarly proved.

By \cite[Proposition 4.1]{tyyNA}, we have
$\jn=JN_{(p,1,s)_{\az}}(\xx)$ with equivalent norms. From this
and Theorem \ref{s2thm1}, we deduce that,
for any $f\in\wz{JN}_{(p,1,s)_{\az}}(\xx)$,
\begin{align*}
\fnjn&\sim \fjn+\|f^{-}\|_{\Rm}\\
&\sim \|f\|_{JN_{(p,1,s)_{\az}}(\xx)}+\|f^{-}\|_{\Rm}
\sim \|f\|_{\wz{JN}_{(p,1,s)_{\az}}(\xx)}<\fz.
\end{align*}
This shows that $\wz{JN}_{(p,1,s)_{\az}}(\xx)\subset\njn $ and
$\|\cdot\|_{\njn}\lesssim\|\cdot\|_{\wz{JN}_{(p,1,s)_{\az}}(\xx)}$.

On the other hand, for any $f\in\njn$, by the H\"{o}lder
inequality, we have
\[\|f\|_{\wz{JN}_{(p,1,s)_{\az}}(\xx)}\lesssim\fnjn.\]
This shows that $\njn\subset\wz{JN}_{(p,1,s)_{\az}}(\xx) $ and
$\|\cdot\|_{\wz{JN}_{(p,1,s)_{\az}}(\xx)}\lesssim\|\cdot\|_{\njn}$,
which completes the proof of Proposition \ref{s3prop4}.
\end{proof}

\begin{remark}
Let $\az=0=s$ and $q\in[1,\fz)$. By Proposition
\ref{s3prop4}, we have
$$\wz{\mathcal{C}}_{0,q,0}(Q_{0})
=\wz{\mathcal{C}}_{0,1,0}(Q_{0})
\quad\text{and}\quad
\|\cdot\|_{\wz{\mathcal{C}}_{0,q,0}(Q_{0})}
\sim\|\cdot\|_{\wz{\mathcal{C}}_{0,1,0}(Q_{0})}.$$
This coincides with \cite[Theorem 3.3]{WS22}.
\end{remark}

\begin{remark}
From Remark \ref{s2rmk5}(iii), we deduce that
$\njnaq=\wz{JN}_{(q,q,s)_{0}}(Q_{0})$ for any given $p\in(1,\fz)$,
$q\in[p,\fz)$, and $s\in\zp$. By this and Proposition
\ref{s3prop4}, we conclude that, for any $p\in(1,\fz)$
and $s\in\zp$,
\[\njnaq=
\begin{cases}
\wz{JN}_{(p,1,s)_{0}}(Q_{0}),&q\in[1,p),\\
\wz{JN}_{(q,q,s)_{0}}(Q_{0}),&q\in [p,\fz).
\end{cases}
\]
Moreover, from Proposition \ref{s3prop2}(i),
it follows that, for any given $p\in[1,\fz]$, $q\in[p,\fz]$,
$s\in\zp$, $\az\in\rr$, and
any given cube $Q_{0}$ of $\rn$,
we have
\[\wz{JN}_{(q,q,s)_{\az}}(Q_{0})\subset \njq \quad
\text{and}\quad
|Q_{0}|^{-\frac{1}{p}}\|\cdot\|_{\wz{JN}_{(q,q,s)_{\az}}(Q_{0})}
\leq |Q_{0}|^{-\frac{1}{q}}\|\cdot\|_{\njq}.\]
\end{remark}

Now, we further investigate the case $s=0$
because, in this case, $P^{(0)}_{Q}(|f|)=|f|_{Q}$
is closely connected with the maximal function $\mathcal{M}_Q(f)$.
To this end, we first establish the following
equivalent characterization.

\begin{proposition}\label{s3prop3}
Let $p,q\in[1,\infty]$ and $\alpha \in\rr$. Then
$f\in\njns$ if and only if $f\in\lql$ and
\begin{equation}\label{s3prop3eq00}
    \sup\lf\{\sum_{i}|Q_{i}|^{1-p\az}\inf_{c\in [0,\fz)}
\|f-c\|_{L^{q}(Q_{i},|Q_{i}|^{-1}dx)}^{p}\r\}^{\frac{1}{p}}<\fz,
\end{equation}
where the infimum is taken over all non-negative real number
and the supremum is taken over all the collections of
interior pairwise disjoint cubes $\qi$ of $\xx$.
Moreover, for any $f\in L^q_\loc(\xx)$, $\|f\|_{\njns}\sim$ the left-hand side
of \eqref{s3prop3eq00} with the positive equivalence constants
independent of $f$.
\end{proposition}

\begin{proof}
Since the case $p=\fz$ was obtained in \cite[Theorem 3.6]{WS22},
it suffices to consider the case $p\in[1,\fz)$.

We first show that ``only if'' part. For any cube $Q\subset\xx$, by
$P^{(0)}_{Q}(|f|)=|f|_{Q}\geq 0$, we conclude that
\[\inf_{c\in [0,\fz)}\|f-c\|_{L^{q}(Q,|Q|^{-1}dx)}
=\inf_{c\in [0,\fz)}
\lf[\fint_{Q}|f-c|^{q}\r]^{\frac{1}{q}}\leq\lf[
\fint_{Q}\lf|f-|f|_{Q}\r|^{q}\r]^{\frac{1}{q}}. \]
Therefore, for any $f\in\njns\subset \lql$, we have
\begin{align}\label{s3prop3eq1}
&\sup\lf\{\sum_{i}|Q_{i}|^{1-p\az}\inf_{c\in [0,\fz)}
\|f-c\|_{L^{q}(Q_{i},|Q_{i}|^{-1}dx)}^{p}\r\}^{\frac{1}{p}}\\
&\quad\leq \sup\lf\{ \sum_{i}|Q_{i}|\lf[ |Q_{i}|^{-\alpha}
\lf\{\fint_{Q_{i}}\lf|f-|f|_{Q}\r|^{q}\r\}^{\frac{1}{q}}\r]
^{p}\r\}^{\frac{1}{p}}
=\|f\|_{\njns}<\fz.\notag
\end{align}

Next, we show that ``if'' part. Let $Q$ denote any cube of
$\xx$. Then, for any given
$f\in L^{q}(Q)$ and for any $c\in [0,\fz)$, we have
\[\lf|c-|f|_{Q}\r|=\lf||f|_{Q}-c\r|\leq \lf|f_{Q}-c\r|
\leq \fint_{Q}\lf|f-c\r|.\]
From this, the Minkowski inequality,
and the H\"{o}lder inequality, it follows that
\begin{align*}
&\lf[\fint_{Q}\lf|f-|f|_{Q}\r|^{q}\r]^{\frac{1}{q}}\\
&\quad\leq  \lf[\fint_{Q}\lf|f- c\r|^{q}\r]^{\frac{1}{q}}
+ \lf[\fint_{Q}\lf| c-|f|_{Q}\r|^{q}\r]^{\frac{1}{q}}
\leq \lf[\fint_{Q}\lf|f- c\r|^{q}\r]^{\frac{1}{q}}
+\fint_{Q}\lf|f- c\r| \\
&\quad\leq \lf[\fint_{Q}\lf|f- c\r|^{q}\r]^{\frac{1}{q}}
+\lf[\fint_{Q}\lf|f- c\r|^{q}\r]^{\frac{1}{q}}
=2\lf[\fint_{Q}\lf|f- c\r|^{q}\r]^{\frac{1}{q}},
\end{align*}
which further implies that
$$\lf[\fint_{Q}\lf|f-|f|_{Q}\r|^{q}\r]^{\frac{1}{q}}
\leq 2\inf_{c\in [0,\fz)}\|f-c\|_{L^{q}(Q,|Q|^{-1}dx)}.$$
By this, we conclude that, for any $f\in \lql$,
\begin{align*}
\|f\|_{\njns}
&=\sup\lf\{ \sum_{i}|Q_{i}|\lf[ |Q_{i}|^{-\alpha}
\lf\{\fint_{Q_{i}}\lf|f-|f|_{Q}\r|^{q}\r\}^{\frac{1}{q}}\r]
^{p}\r\}^{\frac{1}{p}}\\
&\leq\sup\lf\{\sum_{i}|Q_{i}|^{1-p\az}\lf[2\inf_{c\in [0,\fz)}
\|f-c\|_{L^{q}(Q_{i},|Q_{i}|^{-1}dx)}\r]^{p}\r\}^{\frac{1}{p}}\\
&=2\sup\lf\{\sum_{i}|Q_{i}|^{1-p\az}\inf_{c\in [0,\fz)}
\|f-c\|_{L^{q}(Q_{i},|Q_{i}|^{-1}dx)}^{p}\r\}^{\frac{1}{p}}<\fz,
\end{align*}
which implies $f\in \njn$.
This finishes the proof of Proposition \ref{s3prop3}.
\end{proof}

\begin{remark}
\begin{enumerate}[\rm (i)]
\item Proposition \ref{s3prop3} implies that,
for any $c_0\in[0,\fz)$,
$$\|\cdot\|_{\njns}\sim\|\cdot+c_0\|_{\njns}.$$

\item It is still \emph{unclear} whether or not
Proposition \ref{s3prop3} holds true for $s\in\nn$,
namely, whether or not $f\in\njn$ if and only if
$$\sup\lf\{\sum_{i}|Q_{i}|^{1-p\az}
\inf_{\{m\in\mathcal{P}_s(Q):\ m\geq 0\}}
\|f-m\|_{L^{q}(Q_{i},|Q_{i}|^{-1}dx)}^{p}\r\}^{\frac{1}{p}}<\fz.$$
Indeed, the ``if'' part can be similarly proved as
in Proposition \ref{s3prop3}.
But the ``only if'' part is no longer obvious because
$P^{(s)}_Q(|f|)\ge0$ may be false when $s>0$.
For instance, let $Q:=[-1,1]$, $a\in(0,1]$,
and $f(x):=a{\bf 1}_{[0,1/a]}(x)$ for any $x\in Q$.
Then it is easy to find that
$$P^{(1)}_{Q}(|f|)(x)=P^{(1)}_{Q}(f)(x)=\frac{3}{4a}x+\frac{1}{2}$$
and hence $P^{(1)}_{Q}(|f|)(x)<0$ when $x\in[-1,-\frac{2a}{3})$.
\end{enumerate}
\end{remark}

In \cite[Proposition 4]{BMR00}, Bastero et al. showed that
$f\in\BMO (\rn)$ and $f^{-}\in L^{\fz}(\rn)$ if and only if
\[\sup_{Q}\fint_{Q}\lf|f-\mathcal{M}_{Q}(f)\r|^{q}<\fz,\]
where $q\in[1,\fz)$, $\mathcal{M}_{Q}$ is the same as in \eqref{MQ},
and the supremum is taken over all cubes $Q$ of $\xx$.
Correspondingly, we have the following equivalence,
which is the second main result of this article.

\begin{theorem}\label{s3thm1}
Let $1\leq q<p\leq \fz$ and $\az\in[0,\fz)$. Then
$f\in \njns$ if and only if $f\in L^{q}_{\rm loc}(\xx)$ and
\begin{equation}\label{s3thm1eq0}
\sup\lf\{ \sum_{i}|Q_{i}|\lf[ |Q_{i}|^{-\alpha}
\lf\{\fint_{Q_{i}}\lf|f-\mathcal{M}_{Q_{i}}(f)\r|^{q}\r\}
^{\frac{1}{q}}\r]
^{p}\r\}^{\frac{1}{p}}<\fz,
\end{equation}
where the supremum is taken over all the collections of
interior pairwise disjoint cubes $\qi$ of $\xx$.
Moreover, for any $f\in L^q_\loc(\xx)$, $\|f\|_{\njns}\sim$ the left-hand
side of \eqref{s3thm1eq0} with the positive equivalence
constants independent of $f$.
\end{theorem}

\begin{proof}
We first show that ``only if'' part by considering the following
two cases.

Case (i)$_1$ $1<q<p\leq\fz$. In the case,
for any cube $Q$ of $\xx$ and
for any non-negative constant $c$, we have
\begin{equation}\label{s3propeq003}
\lf[\fint_{Q}\lf|f-\mathcal{M}_{Q}(f)\r|^{q}\r]^\frac{1}{q}
\leq \lf[\fint_{Q}\lf|f-c\r|^{q}\r]^\frac{1}{q}
+\lf[\fint_{Q}\lf|c-\mathcal{M}_{Q}(f)\r|^{q}\r]^\frac{1}{q}.
\end{equation}
In addition, from $c\ge 0$, it follows that, for any $x\in Q$,
\begin{align*}
\lf|\mathcal{M}_{Q}(f)(x)-c\r|
&=\lf|\sup_{\{\text{cube }Q_\ast:\ x\in Q_\ast\subset Q\}}
\fint_{Q_\ast} |f(y)|\,dy-c\r|\\
&=\lf|\sup_{\{\text{cube }Q_\ast:\ x\in Q_\ast\subset Q\}}
\fint_{Q_\ast} \lf[|f(y)|-c\r]\,dy\r|\\
&\le\sup_{\{\text{cube }Q_\ast:\ x\in Q_\ast\subset Q\}}
\fint_{Q_\ast} \big||f(y)|-c\big|\,dy\\
&\le\sup_{\{\text{cube }Q_\ast:\ x\in Q_\ast\subset Q\}}
\fint_{Q_\ast} \big|f(y)-c\big|\,dy=\mathcal{M}_{Q}(f-c)(x),
\end{align*}
which, together with the boundedness of $\mathcal{M}_{Q}$
on $L^{q}(Q)$
(see, for instance, \cite[Theorem 2.2]{h01}),
further implies that
\begin{equation}\label{s3propeq004}
\lf[\fint_{Q}\lf|\mathcal{M}_{Q}(f)-c\r|^{q}\r]^\frac{1}{q}
\leq \lf[\fint_{Q}\lf|\mathcal{M}_{Q}(f-c)\r|^{q}\r]^\frac{1}{q}
\leq C_{(q,n)} \lf[\fint_{Q}\lf|f-c\r|^{q}\r]^\frac{1}{q},
\end{equation}
where  the positive constant $C_{(q,n)}$ depends only on both $q$
and $n$. Thus, by \eqref{s3propeq003}
and \eqref{s3propeq004}, we obtain
\[\lf[\fint_{Q}\lf|f-\mathcal{M}_{Q}(f)\r|^{q}\r]^\frac{1}{q}
\leq \lf[1+C_{(q,n)}\r] \lf[\fint_{Q}\lf|f-c\r|^{q}\r]^\frac{1}{q},\]
which, combined with the arbitrariness of $c\ge0$,
further implies that
\begin{align*}
\lf[\fint_{Q}\lf|f-\mathcal{M}_{Q}(f)\r|^{q}\r]^\frac{1}{q}
&\leq \lf[1+C_{(q,n)}\r] \inf_{c\in [0,\fz)}\lf[\fint_{Q}\lf|f-c\r|^{q}\r]^\frac{1}{q}\\
&=\lf[1+C_{(q,n)}\r]\inf_{c\in [0,\fz)}\|f-c\|_{L^{q}(Q,|Q|^{-1}dx)}.
\end{align*}
From this, it follows that, for any $f\in\njns$ and any interior
pairwise disjoint cubes $\qi$
of $\xx$,
\begin{align*}
&\lf\{ \sum_{i}|Q_{i}|\lf[ |Q_{i}|^{-\alpha}
\lf\{\fint_{Q_{i}}\lf|f-\mathcal{M}_{Q_{i}}(f)\r|^{q}\r\}
^{\frac{1}{q}}\r]
^{p}\r\}^{\frac{1}{p}}\\
&\quad \leq\lf[1+C_{(q,n)}\r]\lf[\sum_{i}|Q_{i}|^{1-p\az}
\inf_{c\in [0,\fz)}\|f-c\|_{L^{q}(Q_{i},|Q_{i}|^{-1}dx)}
^{p}\r]^{\frac{1}{p}}.
\end{align*}
By this and \eqref{s3prop3eq1}, we conclude that
\begin{align*}
&\sup\lf\{ \sum_{i}|Q_{i}|\lf[ |Q_{i}|^{-\alpha}
\lf\{\fint_{Q_{i}}\lf|f-\mathcal{M}_{Q_{i}}(f)\r|^{q}\r\}
^{\frac{1}{q}}\r]
^{p}\r\}^{\frac{1}{p}}\\
&\quad\leq \lf[1+C_{(q,n)}\r]\sup\lf[\sum_{i}|Q_{i}|^{1-p\az}
\inf_{c\in [0,\fz)}\|f-c\|_{L^{q}(Q_{i},|Q_{i}|^{-1}dx)}
^{p}\r]^{\frac{1}{p}}\\
&\quad\leq \lf[1+C_{(q,n)}\r]\|f\|_{\njns}<\fz.
\end{align*}

Case (ii)$_1$ $1=q<p\leq\fz$. In the case, let
$r:=\frac{1+\min\{p,3\}}{2}\in(1,p)
$. Then, by the H\"{o}lder inequality, Case (i)$_1$
with $q$ replaced by $r$, and Proposition
\ref{s3prop4}, we have, for any $f\in \wz{JN}_{(p,1,0)_{\az}}(\xx)$
\begin{align*}
&\sup\lf\{ \sum_{i}|Q_{i}|\lf[ |Q_{i}|^{-\alpha}
\lf\{\fint_{Q_{i}}\lf|f-\mathcal{M}_{Q_{i}}(f)\r|\r\}\r]
^{p}\r\}^{\frac{1}{p}}\\
&\quad \leq \sup\lf\{ \sum_{i}|Q_{i}|\lf[ |Q_{i}|^{-\alpha}
\lf\{\fint_{Q_{i}}\lf|f-\mathcal{M}_{Q_{i}}(f)\r|^{r}\r\}
^{\frac{1}{r}}\r]
^{p}\r\}^{\frac{1}{p}}\\
&\quad\leq \lf[1+C_{(q,n)}\r] \|f\|_{\wz{JN}_{(p,r,0)_{\az}}(\xx)}
\lesssim\|f\|_{\wz{JN}_{(p,1,0)_{\az}}(\xx)}<\fz.
\end{align*}
Combining both Cases (i)$_1$ and (ii)$_1$, we find
that the ``only if'' part holds true.

Next, we show that ``if'' part by considering the following
two cases.

Case (i)$_2$ $1=q<p\leq \fz$. In this case, let
$f\in L^{1}_{\rm loc}(\xx)$, $Q$ be a given cube
of $\xx$, $E:=\lf\{x\in Q:\ f(x)\leq f_{Q}\r\}$, and
$F:=\lf\{x\in Q:\ f(x)> f_{Q}\r\}$. The following equality is
trivially true:
\begin{align}\label{E=F}
\int_{E}|f-f_{Q}|=\int_{F}|f-f_{Q}|.
\end{align}
Combining this and the observation
$f(x)\leq f_{Q}\leq \mathcal{M}_{Q}(f)(x)$ for any $x\in E$,
we obtain
\begin{equation}\label{s3propeq005}
\fint_{Q}\lf|f-f_{Q}\r|=\frac{2}{|Q|}\int_{E}\lf|f-f_{Q}\r|
\leq\frac{2}{|Q|}\int_{E}\lf|f-\mathcal{M}_{Q}(f)\r|
\leq 2\fint_{Q}\lf|f-\mathcal{M}_{Q}(f)\r|.
\end{equation}
Moreover, notice that
$\mathcal{M}_{Q}(f)(x)\geq |f(x)|$ for almost every $x\in Q$.
Thus, we have, for almost every $x\in Q$,
\begin{equation*}
0\leq f^{-}(x)\leq \mathcal{M}_{Q}(f)(x)-f^{+}(x)+f^{-}(x)
=\mathcal{M}_{Q}(f)(x)-f(x),
\end{equation*}
which implies that
\begin{equation}\label{s3propeq006}
f^{-}_{Q}= \fint_{Q}|f^{-}|\leq \fint_{Q}\lf|f-\mathcal{M}_{Q}(f)\r|.
\end{equation}
From \eqref{s3propeq005}
and \eqref{s3propeq006}, it follows that
\[
\fint_{Q}\lf|f-|f|_{Q}\r|=\fint_{Q}\lf|f-f_{Q}-2f^{-}_{Q}\r|
\leq \fint_{Q}\lf|f-f_{Q}\r|+2f^{-}_{Q}
\leq4\fint_{Q}\lf|f-\mathcal{M}_{Q}(f)\r|.
\]
By this, we conclude that,
for any interior pairwise disjoint cubes $\qi$ of $\xx$,
\begin{align*}
    &\lf\{ \sum_{i}|Q_{i}|\lf[ |Q_{i}|^{-\alpha}
\lf\{\fint_{Q_{i}}\lf|f-|f|_{Q_{i}}\r|\r\}\r]
^{p}\r\}^{\frac{1}{p}}\\
&\quad\leq 4\lf\{ \sum_{i}|Q_{i}|\lf[ |Q_{i}|^{-\alpha}
\lf\{\fint_{Q_{i}}\lf|f-\mathcal{M}_{Q_{i}}(f)\r|\r\}\r]
^{p}\r\}^{\frac{1}{p}},
\end{align*}
which implies that
\[\|f\|_{\wz{JN}_{(p,1,0)_{\az}}(\xx)}
\leq 4\sup\lf\{ \sum_{i}|Q_{i}|\lf[ |Q_{i}|^{-\alpha}
\lf\{\fint_{Q_{i}}\lf|f-\mathcal{M}_{Q_{i}}(f)\r|\r\}\r]
^{p}\r\}^{\frac{1}{p}}<\fz\]
and hence $f\in \wz{JN}_{(p,1,0)_{\az}}(\xx)$.

Case (ii)$_2$ $1<q<p\leq \fz$. In this case, for any
$f\in L^{q}_{\rm loc}(\xx)$, from Proposition \ref{s3prop4},
Case (i)$_2$,
and the H\"{o}lder inequality, we deduce that
\begin{align*}
&\|f\|_{\njns}\\
&\quad\lesssim\|f\|_{\wz{JN}_{(p,1,0)_{\az}}(\xx)}\\
&\quad\lesssim \sup\lf\{ \sum_{i}|Q_{i}|\lf[ |Q_{i}|^{-\alpha}
\lf\{\fint_{Q_{i}}\lf|f-\mathcal{M}_{Q_{i}}(f)\r|\r\}\r]
^{p}\r\}^{\frac{1}{p}}\\
&\quad\lesssim  \sup\lf\{ \sum_{i}|Q_{i}|\lf[ |Q_{i}|^{-\alpha}
\lf\{\fint_{Q_{i}}\lf|f-\mathcal{M}_{Q_{i}}(f)\r|
^{q}\r\}^{\frac{1}{q}}\r]
^{p}\r\}^{\frac{1}{p}}<\fz
\end{align*}
and hence $f\in \njns$.
Combining Cases (i)$_2$ and (ii)$_2$,
we find that the ``if'' part also holds true.
This finishes the proof of Theorem \ref{s3thm1}.
\end{proof}

\begin{remark}\label{rem-q<p}
We prove Theorem \ref{s3thm1} under the assumptions on $p$, $q$, and $\az$ same as
in Proposition \ref{s3prop4}.
This is reasonable because, based on the following two observations,
it is necessary to apply the equivalence of $\njns$ between $q=1$
and $q\in(1,\fz)$ in the proof of Theorem \ref{s3thm1}:
\begin{enumerate}[\rm(i)]
\item The maximal operator is bounded when $q\in(1,\fz)$, but not bounded when $q=1$;
\item The domination of the mean oscillation of functions
by their maximal functions, namely,
\eqref{s3propeq005}, follows essentially from the linearity of
the integral,
namely, \eqref{E=F}. This linearity holds true only when $q=1$.
\end{enumerate}
Indeed, when proving Theorem \ref{s3thm1},
we flexibly change the index $q$ such that
all these $q$ correspond to the same space,
which is just the conclusion of Proposition \ref{s3prop4}.

\end{remark}

In what follows, we still use $\mathcal{M}$ to denote
the Hardy--Littlewood maximal operator on $\xx$
even when $\xx=Q_0$. Namely, for any $f\in L^1(Q_0)$ and $x\in Q_0$,
$$\mathcal{M}(f)(x)
:=\sup_{\{\text{cube }Q_\ast:\
x\in Q_\ast\subset Q_0\}}\fint_{Q_\ast} |f(y)|\,dy.$$
Based on this, even when $\xx=Q_0$, we still use $[b,\mathcal{M}]$
as in \eqref{commutator} to
denote the commutator generated by $\mathcal{M}$ defined on $Q_0$.

\begin{corollary}\label{s3cor2}
Let $p,q\in[1,\fz]$ and $\az\in \rr$.
\begin{enumerate}[\rm(I)]
\item The following three
statements are mutually equivalent:
\begin{enumerate}[\rm(i)]
\item $f\in \njns;$
\item $f\in \jns$ and $f^{-}\in \Rm;$
\item  $f\in\lql$ and
\begin{equation}\label{s3cor1eq0}
\sup\lf\{\sum_{i}|Q_{i}|^{1-p\az}\inf_{c\in [0,\fz)}
\|f-c\|_{L^{q}(Q_{i},|Q_{i}|^{-1}dx)}^{p}\r\}^{\frac{1}{p}}<\fz,
\end{equation}
where the supremum is taken over all the collections of
interior pairwise disjoint cubes $\qi$ of $\xx$.
\end{enumerate}
Moreover, for any $f\in L^q_\loc(\xx)$, $\|f\|_{\njns}\sim [\|f\|_{\jns}
+\|f^{-}\|_{\Rm}]\sim$ the left-hand side of \eqref{s3cor1eq0}
with the positive equivalence constants independent of $f$.
\item The following two statements are mutually equivalent:
\begin{enumerate}
\item[\rm (iv)]  $f\in L^{q}_{\rm loc}(\xx)$ and
\begin{equation}\label{s3cor1eq1}
\sup\lf\{ \sum_{i}|Q_{i}|\lf[ |Q_{i}|^{-\alpha}
\lf\{\fint_{Q_{i}}\lf|f-\mathcal{M}_{Q_{i}}(f)\r|
^{q}\r\}^{\frac{1}{q}}\r]
^{p}\r\}^{\frac{1}{p}}<\fz,
\end{equation}
where the supremum is taken the same as in {\rm (iii);}
\item[\rm (v)] $f\in\lql$ and
\begin{equation}\label{s3cor1eq2}
\sup\lf\{ \sum_{i}|Q_{i}|\lf[ |Q_{i}|^{-\alpha}
\lf\{\fint_{Q_{i}}\lf|\lf[f,\mathcal{M}\r]({\bf 1}_{Q_{i}})
\r|^{q}\r\}
^{\frac{1}{q}}\r]
^{p}\r\}^{\frac{1}{p}}<\fz,
\end{equation}
where the supremum is taken the same as in {\rm (iii)}.
\end{enumerate}
Moreover, for any $f\in L^q_\loc(\xx)$,
the left-hand sides of both \eqref{s3cor1eq1} and \eqref{s3cor1eq2}
are equivalent with the positive equivalence constants
independent of $f$.
\item If $q<p$ and $\az\ge0$,
then all these five statements in both (I) and (II)
are mutually equivalent and so do the corresponding quantities
with the positive equivalence constants independent of $f$.
\end{enumerate}
\end{corollary}
\begin{proof}
Let $p$, $q$, and $\az$ be the same as in the present corollary.
By Theorem \ref{s2thm1}, we find that (i)
$\Longleftrightarrow$ (ii). From
Proposition \ref{s3prop3}, it follows that
(i) $\Longleftrightarrow$ (iii).
Moreover, Theorem \ref{s3thm1} shows
that (i) $\Longleftrightarrow$ (iv) when $q<p$ and $\az\geq 0$.
Therefore, it suffices to prove (iv) $\Longleftrightarrow$ (v).
Indeed, for any given cube $Q$ of $\xx$ and any $x\in Q$, we have
$$\mathcal{M}_{Q}({\bf 1}_{Q})(x)={\bf 1}_{Q}(x)
\quad\text{and}\quad
\mathcal{M}(f{\bf 1}_{Q})(x)=\mathcal{M}_{Q}(f)(x),$$
and hence, for any $q\in[1,\fz]$,
\begin{align*}
\lf[\fint_{Q}\lf|f-\mathcal{M}_{Q}(f)\r|^{q}\r]^{\frac{1}{q}}
&=\lf[\fint_{Q}\lf|f\mathcal{M}_{Q}({\bf 1}_{Q})
-\mathcal{M}_{Q}(f{\bf 1}_{Q})\r|^{q}
\r]^{\frac{1}{q}}\\
&=\lf[\fint_{Q}\lf|f\mathcal{M}({\bf 1}_{Q})
-\mathcal{M}(f{\bf 1}_{Q})\r|^{q}
\r]^{\frac{1}{q}}\\
&=\lf[\fint_{Q}\lf|\lf[f,\mathcal{M}\r]({\bf 1}_{Q})\r|^{q}\r]
^{\frac{1}{q}}.
\end{align*}
By this, we conclude that
\begin{align*}
&\sup\lf\{ \sum_{i}|Q_{i}|\lf[ |Q_{i}|^{-\alpha}
\lf\{\fint_{Q_{i}}\lf|f-\mathcal{M}_{Q_{i}}(f)\r|
^{q}\r\}^{\frac{1}{q}}\r]
^{p}\r\}^{\frac{1}{p}}\\
&\quad=\sup\lf\{ \sum_{i}|Q_{i}|\lf[ |Q_{i}|^{-\alpha}
\lf\{\fint_{Q_{i}}\lf|\lf[f,\mathcal{M}\r]
({\bf 1}_{Q_{i}})\r|^{q}\r\}
^{\frac{1}{q}}\r]
^{p}\r\}^{\frac{1}{p}},
\end{align*}
which shows the equivalence (iv)
$\Longleftrightarrow$ (v). This finishes
the proof of Corollary \ref{s3cor2}.
\end{proof}

\begin{remark}
Let $p=\fz$ and $\az=0$. In this case, if $q\in(1,\fz)$
and $f\in L^{1}_{\rm loc}(\rn)$ is real-valued, then
the boundedness of $[f,\mathcal{M}]$ on $L^q(\rn)$
is equivalent to that \eqref{s3cor1eq2} holds true,
which was proved by Bastero et al. \cite{BMR00} via
the real interpolation techniques.
\end{remark}

\section{John--Nirenberg-Type Inequality \label{s4}}

In this section, we prove a \emph{good-$\lambda$ inequality}
(namely, Lemma \ref{s4lem2} below) and apply
this good-$\lambda$ inequality to
prove a John--Nirenberg type inequality on $\njq$
(namely, Theorem \ref{s4thm1} below)
via borrowing some ideas from the proof of
\cite[Theorem 4.3]{tyyNA}.
Our main tool is the Calder\'{o}n--Zygmund decomposition.
As an application, we use this John--Nirenberg type inequality
to obtain another proof of Proposition \ref{s3prop4} at
the end of this section.

In what follows, for any given $p\in[1,\fz)$, the
\emph{weak Lebesgue space} $L^{p,\fz}(\xx)$ is defined to be
the set of all the measurable functions $f$ on $\xx$ such that
\[\|f\|_{L^{p,\fz}(\xx)}:=\sup_{\lambda\in(0,\fz)}\lambda
\lf|\lf\{x\in\xx :|f(x)|>\lambda\r\}\r|^{\frac{1}{p}}<\fz.\]
Now, we state the third main result of this article.

\begin{theorem}\label{s4thm1}
Let $p\in(1,\fz)$, $s\in\zp$, $\az \in\rr$, and
$\q$ be a cube of $\rn$. If
$f\in\wz{JN}_{(p,1,s)_{\az}}(\q)$, then
$f-P^{(s)}_{\q}(|f|)\in L^{p,\fz}(\q)$ and  there exists a
positive constant $C_{(n,p,s)}$, depending only on
$n$, $p$, and $s$, but independent of $f$, such that
\begin{align}\label{s4thm1eq1}
\lf\|f-P^{(s)}_{\q}(|f|)\r\|_{L^{p,\fz}(\q)}
\leq C_{(n,p,s)}|\q|^{\az}
\|f\|_{\wz{JN}_{(p,1,s)_{\az}}(\q)}.
\end{align}
\end{theorem}

Recall that, for any cube $Q$ and any $\ell\in\zp$,
\begin{align*}
\mathcal{D}^{(\ell)}_{Q}& :=\Big\{(x_{1},\dots,x_{n})\in\rn:
\text{for any }i\in\{1,\dots,n\},\\
&\qquad x_{i}\in\lf[a_{i}+k_{i}
2^{-\ell}l(Q),
a_{i}+(k_{i}+1)2^{-\ell}l(Q)\r]\ \text{with } k_{i}\in\{0,\dots,2^{\ell}-2\} \\
&\qquad\text{or }x_{i}
\lf.\in\lf[a_{i}+(1-2^{-\ell})l(Q),a_{i}+l(Q)\r]\r\},
\end{align*}
where $l(Q)$ denotes the edge length of $Q$
and $(a_{1},\dots,a_{n})$ is a left and lower vertex of $Q$, which
means that, for any $(x_{1},\dots,x_{n})\in Q$, $x_{i}\geq a_{i}$
for any $i\in\{1,\dots,n\}$. Moreover,
the \emph{dyadic family} $\mathcal{D}_{Q}$
on $Q$ is defined by setting
\[\mathcal{D}_{Q}:=\bigcup_{\ell\in\zp}\mathcal{D}^{(\ell)}_{Q}.\]

In what follows, for any cube $Q$ of $\rn$, $\md$
denotes the \emph{dyadic maximal operator} related to
the dyadic family $\mathcal{D}_{Q}$ on $Q$, namely,
for any $f\in L^{1}(Q)$ and $x\in Q$,
\[\md(f)(x):=\sup_{Q_{(x)}\ni x}\fint_{Q_{(x)}}|f(y)|\,dy,\]
where the supremum is taken over all the dyadic cubes
$Q_{(x)}\in \mathcal{D}_{Q}$ containing $x$. The  following
Calder\'{o}n--Zygmund decomposition, which is just
\cite[p.150, Lemma 1]{S93}, is
needed
in the proof of Theorem \ref{s4thm1}.

\begin{lemma}\label{s4lem1}
Let $\q$ be a cube of $\rn$, $f\in L^{1}(\q)$, and
$\lambda\geq\fint_{\q}|f|$. Then there exist
disjoint dyadic cubes $\{Q_{k}\}_{k}\subset
\mathcal{D}_{\q}$ such that
\begin{enumerate}[\rm(i)]
\item $\{x\in\q:\mdq(f)(x)>\lambda\}=\bigcup_{k}
Q_{k}$;
\item
\(\lambda<\fint_{Q_{k}}|f|\leq2^{n} \lambda\)
\ \ for any $k\in\nn$;
\item \(|\{x\in\q:\mdq(f)(x)>\lambda\}|\leq
\frac{1}{\lambda}\int_{\{x\in\q:\mdq(f)(x)>\lambda\}}
|f|;\)
\item $f(x)\leq\lambda$ for almost every
$x\in\q\setminus \bigcup_{k}Q_{k}$.
\end{enumerate}
\end{lemma}

Now, we prove Theorem \ref{s4thm1} via beginning with the
following \emph{good-$\lambda$ inequality}. We borrow some
ideas from the proof of  \cite[Lemma 4.5]{ABKY11}
(see also \cite[Lemma 4.6]{tyyNA})
with suitable
modifications.
\begin{lemma}\label{s4lem2}
Let $p\in(1,\fz)$, $s\in \zp$, $\c\in[1,\fz)$ be the same
as in \eqref{s1eq2},
$\theta\in(0,2^{-n}\c^{-1})$, $\q$ be a
cube of $\rn$, and $f\in\njnqq$. Then, for any
$\lambda\geq\frac{1}{\theta}\fint_{\q}|f-P^{(s)}_{\q}(|f|)|$,
\begin{align}\label{s4lem2eq1}
&\lf|\lf\{x\in\q:\mdq\lf(f-P^{(s)}_{\q}(|f|)\r)(x)>\lambda\r\}\r|\\
&\quad \leq\frac{1+\c}{1-2^{n}\theta\c}
\frac{\|f\|_{\njnqq}}{\lambda}
\lf|\lf\{x\in\q:\mdq\lf(f-P^{(s)}_{\q}(|f|)\r)(x)
>\theta\lambda\r\}\r|
^{\frac{1}{p'}},\notag
\end{align}
where
$p'$ denotes the \emph{conjugate index} of $p$, that is,
$p'$ satisfies $1/p+1/p'=1$.
\end{lemma}

\begin{proof}
Let all symbols be the same as in the present lemma.
Applying Lemma \ref{s4lem1} to $f-P^{(s)}_{\q}(|f|)$ on $\q$
at height
$\theta\lambda$, we find interior pairwise disjoint dyadic
cubes $\{Q_{k}\}_{k}\subset\mathcal{D}_{\q}$ such that
$$\lf\{x\in\q:\mdq\lf(f-P^{(s)}_{\q}(|f|)\r)(x)
>\theta\lambda\r\}=\bigcup_{k}
Q_{k}.$$
Since $\theta<2^{-n}\c^{-1}<1$, it follows that
\begin{equation*}
\lf\{x\in\q:\mdq\lf(f-P^{(s)}_{\q}(|f|)\r)(x)>\lambda\r\}
\subset\lf\{x\in\q:\mdq\lf(f-P^{(s)}_{\q}(|f|)\r)(x)
>\theta\lambda\r\}
\end{equation*}
and hence
\begin{align}\label{s4lem2eq2}
&\lf\{x\in\q:\mdq\lf(f-P^{(s)}_{\q}(|f|)\r)(x)>\lambda\r\}\\
&\quad=\bigcup_{k}\lf\{x\in Q_{k}:\mdq\lf(f-P^{(s)}_{\q}(|f|)\r)(x)>
\lambda\r\}.\notag
\end{align}

We now claim that, for any $k\in \mathbb{N}$,
\begin{align}\label{s4lem2eq3}
&\lf\{x\in Q_{k}:\mdq\lf(f-P^{(s)}_{\q}(|f|)\r)(x)>\lambda\r\}\\
&\quad\subset\lf\{x\in Q_{k}:\mdq\lf(\lf[f-P^{(s)}_{Q_{k}}(|f|)
\r]{\bf1}_{Q_{k}}\r)(x)>\frac{[1-2^{n}\theta\c]\lambda}{1+\c}\r\}.\notag
\end{align}
Indeed, for any $x\in Q_{k}$ with $\mdq\lf(f-P^{(s)}_{\q}
(|f|)\r)(x)>\lambda$, by Lemma \ref{s4lem1}, we know that there
exists a dyadic cube $Q_{(x)}\ni x$ in $\mathcal{D}
_{\q}$ such that
\begin{equation}\label{s4lem2eq4}
\fint_{Q_{(x)}}\lf|f-P^{(s)}_{\q}(|f|)\r|>\lambda.
\end{equation}
Since $Q_{k}$ is the maximal dyadic cube satisfying
$\fint_{Q}|f-P^{(s)}_{Q}(|f|)|>\theta\lambda$,
then we have $Q_{(x)}\subset Q_{k}$. From this and
\eqref{s4lem2eq4}, it follows that
\[\mdq\lf(\lf[f-P^{(s)}_{\q}(|f|)\r]{\bf1}_{Q_{k}}\r)(x)
\geq\fint_{Q_{(x)}}\lf|f-P^{(s)}_{\q}(|f|)\r|>\lambda.\]
By this,
$P^{(s)}_{Q_{k}}(P^{(s)}_{Q_{k}}(|f|))=
P^{(s)}_{Q_{k}}(|f|)$, $P^{(s)}_{Q_{k}}(P^{(s)}_{Q_{0}}(|f|))
=P^{(s)}_{Q_{k}}(|f|)$, the linearity of $P^{(s)}_{Q_{k}}$,
\eqref{s1eq2}, and Lemma \ref{s4lem1}{\rm(ii)},
we conclude that, for any $x\in Q_{k}$,
\begin{align*}
\lambda
&<\mdq\lf(\lf[f-P^{(s)}_{\q}(|f|)\r]{\bf1}_{Q_{k}}\r)(x)
=\sup_{Q\ni x}\fint_{Q}\lf|\lf(\lf[f-P^{(s)}_{\q}(|f|)\r]{\bf1}
_{Q_{k}}\r)(y)\r|\,dy\\
&\leq\sup_{Q\ni x}\lf[\fint_{Q}\lf|f(y)-P^{(s)}_{Q_{k}}
(|f|)(y)\r|{\bf1}_{Q_{k}}(y)\,dy
+\fint_{Q}\lf|P^{(s)}_{Q_{k}}(|f|)(y)-P^{(s)}_{Q_{k}}(f)(y)\r|
{\bf1}_{Q_{k}}(y)\,dy\r.\\
&\quad\lf.+\fint_{Q}\lf|P^{(s)}_{Q_{k}}(f)(y)-P^{(s)}_{\q}(|f|)
(y)\r|{\bf1}_{Q_{k}}(y)\,dy\r]\\
&=\sup_{Q\ni x}\lf[\fint_{Q}\lf|f(y)-P^{(s)}_{Q_{k}}
(|f|)(y)\r|{\bf1}_{Q_{k}}(y)\,dy
+\fint_{Q}\lf|P^{(s)}_{Q_{k}}(f-P^{(s)}_{Q_{k}}(|f|))(y)\r|
{\bf1}_{Q_{k}}(y)\,dy\r.\\
&\quad\lf.+\fint_{Q}\lf|P^{(s)}_{Q_{k}}(f-P^{(s)}_{\q}(|f|))(y)
\r|{\bf1}_{Q_{k}}(y)\,dy\r]\\
&\leq\mdq\lf(\lf[f-P^{(s)}_{Q_{k}}(|f|)
\r]{\bf1}_{Q_{k}}\r)(x)+\c\fint_{Q_{k}}\lf|f-P^{(s)}_{Q_{k}}
(|f|)\r|
+\c\fint_{Q_{k}}\lf|f-P^{(s)}_{\q}(|f|)\r|\\
&\leq\lf[1+\c\r]\mdq\lf(\lf[f-P^{(s)}_{Q_{k}}(|f|)
\r]{\bf1}_{Q_{k}}\r)(x)+2^{n}\theta\c\lambda.
\end{align*}
This shows that the above claim \eqref{s4lem2eq3} holds true.

Next, we prove that, for any $k$,
\begin{align}\label{s4lem2eq5}
&\lf|\lf\{x\in Q_{k}:\mdq\lf(f-P^{(s)}_{\q}(|f|)\r)(x)
>\lambda\r\}\r|\\
&\quad\leq\frac{1+\c}{[1-2^{n}\theta\c]\lambda}
\int_{Q_{k}}\lf|f-P^{(s)}_{Q_{k}}(|f|)\r|.\notag
\end{align}
Indeed, if $\fint_{Q_{k}}|f-P^{(s)}_{Q_{k}}(|f|)|>
\frac{[1-2^{n}\theta\c]\lambda}{1+\c}$, then we
have $|Q_{k}|<\frac{1+\c}{[1-2^{n}
\theta\c]\lambda}\int_{Q_{k}}|f-P^{(s)}_{Q_{k}}(|f|)|$
due to $\theta<2^{n}\c$. Thus,
\[\lf|\lf\{x\in Q_{k}:\mdq\lf(f-P^{(s)}_{\q}(|f|)\r)(x)
>\lambda\r\}\r|\leq |Q_{k}|<\frac{1+\c}{[1-2^{n}
\theta\c]\lambda}
\int_{Q_{k}}\lf|f-P^{(s)}_{Q_{k}}(|f|)\r|,\]
and hence \eqref{s4lem2eq5} holds true in this case.
If $\fint_{Q_{k}}|f-P^{(s)}_{Q_{k}}(|f|)|\leq
\frac{[1-2^{n}\theta\c]\lambda}{1+\c}$, then,
applying Lemma \ref{s4lem1} to $f-P^{(s)}_{Q_{k}}(|f|)$
on $Q_{k}$ at height $\frac{[1-2^{n}\theta\c]\lambda}{1+\c}$,
we obtain
\begin{align*}
&\lf|\lf\{x\in Q_{k}:\mathcal{M}^{(\rm d)}_{Q_{k}}
\lf(\lf[f-P^{(s)}_{Q_{k}}(|f|)\r]{\bf1}_{Q_{k}}\r)(x)
>\frac{[1-2^{n}\theta\c]\lambda}{1+\c}\r\}\r|\\
&\quad\leq\frac{1+\c}{[1-2^{n}\theta\c]\lambda}
\int_{Q_{k}}\lf|f-P^{(s)}_{Q_{k}}(|f|)\r|.
\end{align*}
From this, $\mathcal{M}^{(\rm d)}_{Q_{0}}([\cdots]{\bf1}_{Q_{k}})=
\mathcal{M}^{(\rm d)}_{Q_{k}}([\cdots]{\bf1}_{Q_{k}})$,
and \eqref{s4lem2eq3}, we deduce that
\begin{align*}
& \lf|\lf\{x\in Q_{k}:\mdq\lf(f-P^{(s)}_{\q}(|f|)\r)(x)
>\lambda\r\}\r|\\
&\quad\leq\lf|\lf\{x\in Q_{k}:\ \mdq\lf(\lf[f-P^{(s)}_{Q_{k}}(|f|)
\r]{\bf1}_{Q_{k}}\r)(x)>\frac{[1-2^{n}\theta\c]\lambda}{1+\c}
\r\}\r|\\
&\quad=\lf|\lf\{x\in Q_{k}:\ \mathcal{M}^{(\rm d)}_{Q_{k}}
\lf(\lf[f-P^{(s)}_{Q_{k}}(|f|)\r]{\bf1}_{Q_{k}}\r)(x)
>\frac{[1-2^{n}\theta\c]\lambda}{1+\c}\r\}\r|\\
&\quad\leq\frac{1+\c}{[1-2^{n}\theta\c]\lambda}
\int_{Q_{k}}\lf|f-P^{(s)}_{Q_{k}}(|f|)\r|.
\end{align*}
Thus, \eqref{s4lem2eq5} holds true. From \eqref{s4lem2eq2},
\eqref{s4lem2eq5}, the H\"{o}lder inequality, and the construction
of $\{Q_{k}\}_{k}$, we deduce that
\begin{align*}
&\lf|\lf\{x\in\q:\mdq\lf(f-P^{(s)}_{\q}(|f|)\r)(x)>\lambda\r\}\r|\\
&\quad\leq\sum_{k}\lf|\lf\{x\in Q_{k}:\mdq\lf(f-P^{(s)}_{\q}
(|f|)\r)(x)>\lambda\r\}\r|\\
&\quad\leq\sum_{k}\frac{[1+\c]}
{[1-2^{n}\theta\c]\lambda}
|Q_{k}|^{\frac1{p'}}
\lf[|Q_{k}|^{\frac{1}{p}-1}
\int_{Q_{k}}\lf|f-P^{(s)}_{Q_{k}}(|f|)\r|\r]\\
&\quad\leq\frac{1+\c}{[1-2^{n}\theta\c]\lambda}
\lf(\sum_{k}|Q_{k}|\r)^{\frac{1}{p'}}\lf\{\sum_{k}|Q_{k}|
^{1-p}\lf[\int_{Q_{k}}\lf|f-P^{(s)}_{Q_{k}}(|f|)\r|\r]^{p}\r\}
^\frac{1}{p}\\
&\quad\leq\frac{1+\c}{[1-2^{n}\theta\c]\lambda}
\lf|\lf\{x\in\q:\mdq\lf(f-P^{(s)}_{\q}(|f|)\r)(x)>
\theta \lambda\r\}\r|^{\frac{1}{p'}}\|f\|_{\njnqq},
\end{align*}
which shows that \eqref{s4lem2eq1} holds true. This finishes
the proof of Lemma \ref{s4lem2}.
\end{proof}

\begin{proof}[Proof of Theorem \ref{s4thm1}]
We first prove \eqref{s4thm1eq1} for $\az=0$. Let $\theta:=
2^{-(n+1)}\c^{-1}$, where $\c$ is the same as \eqref{s1eq2},
and let $\eta:=\frac{\|f\|_{\wz{JN}_{(p,1,s)_{0}}(\q)}}{\theta |\q|
^{1/p}} $. We show that
\begin{align*}\lf\|f-P^{(s)}_{\q}(|f|)\r\|_{L^{p,\fz}(\q)}
&=\sup_{\lambda\in(0,\fz)}
\lambda\lf|\lf\{x\in\q :\ \lf|f(x)-P^{(s)}_{\q}(|f|)(x)\r|>\lambda\r\}\r|^
{\frac{1}{p}}\\
&\lesssim \|f\|_{\njnqq}
\end{align*}
by considering the following two cases.

Case (i) $\lambda\leq\eta$, namely, $\lambda\leq
\frac{\|f\|_{\wz{JN}_{(p,1,s)_{0}}(\q)}}{\theta |\q|
^{1/p}}$.
In this case, $\lambda|\q|^{\frac{1}{p}}\leq 2^{n+1}\c\|f\|
_{\njnqq}$
and hence
\begin{align*}
&\sup_{\lambda\in (0,\eta]}\lambda\lf|\lf\{x\in\q
:\ \lf|f(x)-P^{(s)}
_{\q}(|f|)(x)\r|>\lambda\r\}\r|^{\frac{1}{p}}\\
&\quad\leq \sup_{\lambda
\in (0,\eta]}\lf[\lambda|\q|^{\frac{1}{p}}\r]\leq 2^{n+1}\c
\|f\|_{\njnqq},
\end{align*}
which is the desired estimate in this case.

Case (ii) $\lambda>\eta$. In this case, by the definition
of $\|f\|_{\njnqq}$ and $\theta<1$, we have
\[\lambda>\eta=\frac{\|f\|_{\wz{JN}_{(p,1,s)_{0}}(\q)}}{\theta |\q|
^{1/p}}\geq\frac{|\q|^{1/p}\fint_{\q}|f-P^{(s)}
_{\q}(|f|)|}{\theta |\q|^{1/p}}>\fint_{\q}\lf|f-P^{(s)}_
{\q}
(|f|)\r|.\]
Next, we show that
\begin{equation}\label{s4thm1eq2}
\sup_{\lambda\in(\eta,\fz)}\lambda\lf|\lf\{x\in\q :\ \mdq
\lf(f-P^{(s)}_{\q}(|f|)\r)(x)>\lambda\r\}\r|^{\frac{1}{p}}
\lesssim\|f\|_{\njnqq}.
\end{equation}
Let $j_{0}$ be the smallest non-negative integer such that
$\theta^{-j}\eta<\lambda$. By \eqref{s4lem2eq1},
$\theta\lambda\leq\theta^{-j_{0}}\eta<\lambda$, and Lemma
\ref{s4lem1}{(iii)}, we have
\begin{align*}
&\lf|\lf\{x\in\q :\ \mdq\lf(f-P^{(s)}_{\q}(|f|)\r)(x)>\lambda\r\}
\r|\\
&\quad\leq\lf|\lf\{x\in\q :\ \mdq\lf(f-P^{(s)}_{\q}(|f|)\r)(x)
>\theta^{-j_{0}}\eta\r\}\r|\\
&\quad\leq\frac{D_0}{\theta^{-j_{0}}\eta}
\lf|\lf\{x\in\q :\ \mdq\lf(f-P^{(s)}_{\q}(|f|)\r)(x)
>\theta^{-j_{0}+1}\eta\r\}\r|^{(p')^{-1}}\\
&\quad\leq\frac{D_0}{\theta^{-j_{0}}\eta}
\lf\{\frac{D_0}{\theta^{-j_{0}+1}\eta}\r\}^
{(p')^{-1}}\\
&\qquad\times \lf|\lf\{x\in\q :\mdq\lf(f-P^{(s)}_{\q}(|f|)\r)(x)
>\theta^{-j_{0}+2}\eta\r\}\r|^{(p')^{-2}}\\
&\quad\leq\frac{D_0}{\theta^{-j_{0}}\eta}
\lf\{\frac{D_0}{\theta^{-j_{0}+1}\eta}\r\}^
{(p')^{-1}}\times\cdots\\
&\qquad\times\lf\{\frac{D_0}
{\theta^{-1}\eta}\r\}^{(p')^{-j_{0}+1}}\lf|\lf\{x\in\q :\mdq\lf(f-P^{(s)}_{\q}(|f|)\r)(x)
>\eta\r\}\r|^{(p')^{-j_{0}}}\\
&\quad\leq\frac{D_0}{\theta\lambda}
\lf\{\frac{D_0}{\theta^{2}\lambda}\r\}^
{(p')^{-1}}\times\cdots\\
&\qquad\times\lf\{\frac{D_0}
{\theta^{j_{0}}\lambda}\r\}^{(p')^{-j_{0}+1}}
\lf(\frac{1}{\eta}\int_{\q}\lf|f-P^{(s)}_{\q}(|f|)\r|\r)
^{(p')^{-j_{0}}}\\
&\quad=\lf(\frac{1}{\theta}\r)^{1+2(p')^{-1}+\cdots+
j_{0}(p')^{-j_{0}+1}}\lf\{\frac{D_0}
{\lambda}\r\}^{1+(p')^{-1}+(p')^{-2}+\cdots+(p')^{-j_{0}+1}}\\
&\qquad\times\lf[\frac{1}{\eta}\int_{\q}\lf|f-P^{(s)}
_{\q}(|f|)\r|\r]^{(p')^{-j_{0}}},
\end{align*}
where $D_0:=2[1+\c]\|f\|_{\njnqq}$.
Observe that $|\q|^{\frac{1}{p}}\fint_{\q}|f-P^{(s)}
_{\q}(|f|)|\leq\|f\|_{\njnqq}$. We obtain
\[\frac{1}{\eta}\int_{\q}\lf|f-P^{(s)}_{\q}(|f|)\r|=
\frac{\theta|\q|^{\frac{1}{p}}}{\|f\|_{\njnqq}}
\int_{\q}\lf|f-P^{(s)}_{\q}(|f|)\r|\leq \theta|\q|.\]
From this, $1+2(p')^{-1}+\cdots+j_{0}(p')^{-j_{0}+1}=
p^{2}[1-(p')^{-j_{0}}]-pj_{0}(p')^{j_{0}}\leq p^{2}$, and
\[1+(p')^{-1}+(p')^{-2}+\cdots+(p')^{-j_{0}+1}
=p[1-(p')^{-j_{0}}],\]
we deduce that
\begin{align}\label{3eq7}
&\lf|\lf\{x\in\q :\ \mdq\lf(f-P^{(s)}_{\q}(|f|)\r)(x)>
\lambda\r\}\r|\\
&\quad\leq\lf(\frac{1}{\theta}\r)^{p^{2}}\lf\{\frac{2[1+\c]
\|f\|_{\njnqq}}{\lambda}\r\}^{p[1-(p')^{-j_{0}}]}(\theta|\q|)
^{(p')^{-j_{0}}}\notag\\
&\quad=\lf\{2\lf[1+\c\r]\r\}^{p[1-(p')^{-j_{0}}]}\lf(\frac{1}{\theta}\r)
^{p^{2}-(p')^{-j_{0}}}\lf[\frac{\|f\|_{\njnqq}}{\lambda}\r]^{p}
\lf[\frac{\lambda|\q|^{\frac{1}{p}}}{\|f\|_{\njnqq}}\r]
^{p(p')^{-j_{0}}}\notag\\
&\quad\leq\lf\{2\lf[1+\c\r]\r\}^{p}\lf(\frac{1}{\theta}\r)^{p^{2}}
\lf[\frac{\|f\|_{\njnqq}}{\lambda}\r]^{p}
\lf[\frac{\lambda|\q|^{\frac{1}{p}}}{\|f\|_{\njnqq}}\r]
^{p(p')^{-j_{0}}}.\notag
\end{align}
By the definitions of both $\eta$ and $j_{0}$, we have
\begin{equation}\label{s4thm1eq3}
\frac{\lambda|\q|^{1/p}}{\|f\|_{\njnqq}}=\frac{\lambda}
{\theta\eta}\leq\frac{\theta^{-j_{0}-1}\eta}{\theta\eta}
=\theta^{-j_{0}-2}.
\end{equation}
We now claim that, for any $j\in\mathbb{N}$,
\begin{equation}\label{s4thm1eq4}
(j+2)(p')^{-j}\leq\max\{p,2\}.
\end{equation}
Indeed, for any $j\in\zp$, let $F(j):=(j+2)(p')^{-j}.$
Then $F$ attains its maximal value at some $j_{1}\in\zp$.
If $j_{1}=0$, then $(j_{1}+2)(p')^{-j_{1}}=2\leq\max\{p,2\}$.
If $j_{1}\in\mathbb{N}$, then
\[\frac{F(j_{1}-1)}{F(j_{1})}=\frac{(j_{1}+1)(p')^{-j_{1}+1}}
{(j_{1}+2)(p')^{-j_{1}}}\leq 1,\]
which implies that $j_{1}+2\leq\frac{1}{p'-1}+1=p$ and hence
$(j_{1}+2)(p')^{-j_{1}}\leq p\leq\max\{p,2\}$. This proves
\eqref{s4thm1eq4}. Therefore, by \eqref{3eq7},
\eqref{s4thm1eq3}, \eqref{s4thm1eq4},
and $\theta=2^{-(n+1)}\c^{-1}$, we conclude that
\begin{align*}
&\lf|\lf\{x\in\q :\ \mdq\lf(f-P^{(s)}_{\q}(|f|)\r)(x)>
\lambda\r\}\r|\\
&\quad\leq 2^{p}\lf[1+\c\r]^{p}\lf(\frac{1}{\theta}\r)^{p^{2}
-p(j_{0}+2)(p')^{-j_{0}}}\lf[\frac{\|f\|_{\njnqq}}{\lambda}
\r]^{p}\\
&\quad\leq 2^{p}\lf[1+\c\r]^{p}\lf(\frac{1}{\theta}\r)
^{p^{2-p\max \{p, 2\}}}\lf[\frac{\|f\|_{\njnqq}}{\lambda}\r]
^{p}\\
&\quad=2^{p+(n+1)(p^{2}-p\max\{p,2\})}\c^{p^{2}-p\max\{p,2\}}
\lf[1+\c\r]^{p}\lf[\frac{\|f\|_{\njnqq}}{\lambda}\r]^{p}.
\end{align*}
This implies that \eqref{s4thm1eq2} holds true. Moreover,
using Lemma
\ref{s4lem1}{\rm{(iv)}}, we find that
\[\lf\{x\in\q:\ \lf|f-P^{(s)}_{\q}(|f|)\r|>\lambda\r\}\subset\lf\{x\in\q:\
\mdq\lf(f-P^{(s)}_{\q}(|f|)\r)(x)>\lambda\r\}.\]
From this and \eqref{s4thm1eq2}, it follows that
\[\sup_{\lambda\in(\eta,\fz)}\lf[\lambda
\lf|\{x\in\q:\ |f-P^{(s)}_{\q}(|f|)|>\lambda\}\r|^{\frac{1}{p}}\r]
\lesssim\|f\|_{\njnqq}.\]
Therefore, \eqref{s4thm1eq1} for $\az=0$ holds true by
combining Case {\rm (i)} and Case {\rm (ii)} and
letting
\[C_{(n,p,s)}:=\max\lf\{2^{n+1}\c,2^{p+(n+1)(p^{2}-
p\max\{p,2\})}\c^{p^{2}-p\max\{p,2\}}\lf[1+\c\r]^{p}\r\}.\]

Finally, for any $\az\in[0,\fz)$, by \eqref{s4thm1eq1} with
$\az=0$, we find that
\begin{align*}
&\lf\|f-P^{(s)}_{\q}(|f|)\r\|_{L^{p,\fz}(\q)}\\
&\quad\leq C_{(n,p,s)}\|f\|_{\wz{JN}_{(p,1,s)_{0}}(\q)}
=C_{(n,p,s)}\sup\lf\{\sum_{i}|Q_{i}|\lf[\fint_{Q_{i}}
\lf|f-P^{(s)}_{Q_{i}}(|f|)\r|\r]^{p}\r\}^{\frac{1}{p}}\\
&\quad\leq C_{(n,p,s)}|\q|^{\az}\sup\lf\{\sum_{i}|Q_{i}|
\lf[|Q_{i}|^{-\az}\lf\{\fint_{Q_{i}}
\lf|f-P^{(s)}_{Q_{i}}(|f|)\r|\r\}\r]^{p}\r\}^{\frac{1}{p}}\\
&\quad\leq C_{(n,p,s)}|\q|^{\az}\|f\|_{\wz{JN}_{(p,1,s)_{\az}}(\q)}.
\end{align*}
This finishes the proof of Theorem \ref{s4thm1}.
\end{proof}

As an application of Theorem \ref{s4thm1}, we give
another proof of Proposition \ref{s3prop4} as follows.

\begin{proof}[Another Proof of Proposition \ref{s3prop4}]
We only consider the case $p<\fz$
because the case $p=\fz $ can be similarly proved.
For any $f\in\njn$ with $1\leq q<p< \fz$ and for any $\az\in\rr$,
by the H\"{o}lder inequality, we obtain
\begin{align*}
\|f\|_{\njnq}&=\sup\lf\{\sum_{i}|Q_{i}|
\lf[|Q_{i}|^{-\az}\lf\{\fint_{Q_{i}}
\lf|f-P^{(s)}_{Q_{i}}(|f|)\r|\r\}\r]^{p}\r\}^{\frac{1}{p}}\\
&\leq \sup\lf\{\sum_{i}|Q_{i}|
\lf[|Q_{i}|^{-\az}\lf\{\fint_{Q_{i}}
\lf|f-P^{(s)}_{Q_{i}}(|f|)\r|^{q}\r\}^{\frac{1}{q}}
\r]^{p}\r\}^{\frac{1}{p}}\\
&=\|f\|_{\njn}.
\end{align*}
This shows that $\njn\subset\njnq$.

On the other hand, for any $1\leq q<p<\fz$ and any
cube $Q$ of $\rn$, by the embedding $L^{p,\fz}(Q)
\subset L^{q}(Q)$ (see, for instance,
\cite[p.\,14, Exercises 1.1.11]{G14}) and
Theorem \ref{s4thm1}, we have
\begin{equation}\label{s4APeq4eq1}
\lf[\fint_{Q}\lf|f-P^{(s)}_{Q}(|f|)\r|^{q}\r]^{\frac{1}{q}}
\lesssim
|Q|^{-\frac{1}{p}}\lf\|f-P^{(s)}_{Q}(|f|)\r\|_{L^{p,\fz}(Q)}
\lesssim |Q|^{-\frac{1}{p}}\|f\|_{\wz{JN}_{(p,1,s)_{0}}(Q)}.
\end{equation}
Now, for any given interior pairwise disjoint cubes
$\qi$ of $\xx$, from the definition of
$\wz{JN}_{(p,1,s)_{0}}(Q_{i})$, we deduce that,
for any $i$, there exist
interior pairwise disjoint cubes $\{Q_{i,j}\}_{j}$ of $Q_{i}$
such that
\begin{equation}\label{s4APeq4eq2}
\wz{JN}_{(p,1,s)_{0}}(Q_{i})\lesssim \lf\{\sum_{j}|Q_{i,j}\lf[
\fint_{Q_{i,j}}\lf|f-P^{(s)}_{Q_{i,j}}(|f|)\r|\r]^{p}\r\}^
{\frac{1}{p}}.
\end{equation}
By \eqref{s4APeq4eq1}, \eqref{s4APeq4eq2}, and $\az\in[0,\fz)$,
we obtain
\begin{align*}
&\sum_{i}|Q_{i}| \lf[|Q_{i}|^{-\az}\lf\{\fint_{Q_{i}}
\lf|f-P^{(s)}_{Q_{i}}(|f|)\r|^{q}\r\}^{\frac{1}{q}}\r]^{p}\\
&\quad\lesssim \sum_{i}|Q_{i}|\lf[|Q_{i}|^{-\az}|Q_{i}|^
{-\frac{1}{p}}\|f\|_{\wz{JN}_{(p,1,s)_{0}}(Q_{i})}\r]^{p}\\
&\quad\sim \sum_{i}|Q_{i}|^{-\az p}\|f\|^{p}_{\wz{JN}_{(p,1,s)
_{0}}(Q_{i})}\lesssim \sum_{i}|Q_{i}|^{-\az p}\sum_{j}|Q_{i,j}|
\lf[\fint_{Q_{i,j}}\lf|f-P^{(s)}_{Q_{i,j}}(|f|)\r|\r]^{p}\\
&\quad \lesssim \sum_{i}\sum_{j}|Q_{i,j}|\lf[|Q_{i,j}|
^{-\az}\fint
_{Q_{i,j}}\lf|f-P^{(s)}_{Q_{i,j}}(|f|)\r|\r]^{p}\lesssim \|f\|^
{p}_{\wz{JN}_{(p,1,s)_{0}}(Q_{i})}.
\end{align*}
This implies that $\|f\|_{\njn}\lesssim\|f\|_{\njnq}$ and hence
$\njnq\subset\njn$.

To sum up, $\njn=\njnq$ and
$\|\cdot\|_{\njn}\sim\|\cdot\|_{\wz{JN}_{(p,1,s)_{\az}}(\xx)}$.
This finishes the proof of Proposition \ref{s3prop4}.
\end{proof}

\bigskip

\noindent  Pingxu Hu and Dachun Yang (Corresponding author)

\medskip

\noindent  Laboratory of Mathematics and Complex Systems
(Ministry of Education of China),
School of Mathematical Sciences, Beijing Normal University,
Beijing 100875, The People's Republic of China

\smallskip

\noindent {\it E-mails}:
\texttt{pxhu@mail.bnu.edu.cn} (P. Hu)

\noindent\phantom{{\it E-mails:}}
\texttt{dcyang@bnu.edu.cn} (D. Yang)

\bigskip

\noindent Jin Tao

\medskip

\noindent Hubei Key Laboratory of Applied Mathematics,
Faculty of Mathematics and Statistics,
Hubei University, Wuhan 430062, The People's Republic of China

\smallskip

\noindent {\it E-mail}: \texttt{jintao@mail.bnu.edu.cn}

\smallskip


\begin{thebibliography}{99}
\bibitem{ABKY11}
D. Aalto, L. Berkovits, O. E. Kansanen and H. Yue,
John--Nirenberg lemmas for a doubling measure,
Studia Math. 204 (2011), 21-37.

\vspace{-0.3cm}

\bibitem{BMR00}
J. Bastero, M. Milman and F. Ruiz,
Commutators for the maximal and sharp functions,
Proc. Amer. Math. Soc. 128 (2000),
3329-3334.

\vspace{-0.3cm}

\bibitem{be20}
O. Blasco and C. Espinoza-Villalva,
The norm of the characteristic function of a set in the
John--Nirenberg space of exponent $p$,
Math. Methods Appl. Sci. 43 (2020), 9327-9336.

\vspace{-0.3cm}

\bibitem{BBP15}
J. Bourgain, H. Brezis and P. Mironescu,
A new function space and applications,
J. Eur. Math. Soc. (JEMS). 17 (2015),
2083-2101.

\vspace{-0.3cm}

\bibitem{C64}
S. Campanato,
Propriet\`{a} di una famiglia di spazi funzionali,
Ann. Scuola Norm. Sup. Pisa Cl. Sci. (3). 18 (1964),
137-160.

\vspace{-0.3cm}

\bibitem{cds99}
D.-C. Chang, G. Dafni and E. M. Stein,
Hardy spaces, BMO, and boundary value problems for the Laplacian on a smooth domain in
$\mathbb{R}^n$,
Trans. Amer. Math. Soc. 351 (1999), 1605-1661.

\vspace{-0.3cm}

\bibitem{cl99}
D.-C. Chang and S.-Y. Li,
On the boundedness of multipliers, commutators and
the second derivatives of Green's operators on $H^1$ and BMO,
Ann. Scuola Norm. Sup. Pisa Cl. Sci. (4) 28 (1999), 341-356.

\vspace{-0.3cm}

\bibitem{cs06}
D.-C. Chang and C. Sadosky,
Functions of bounded mean oscillation,
Taiwanese J. Math. 10 (2006), 573-601.

\vspace{-0.3cm}

\bibitem{DHKY18}
G. Dafni, T. Hyt\"onen, R. Korte and H. Yue,
The space $JN_p$: nontriviality and duality,
J. Funct. Anal. 275 (2018), 577-603.

\vspace{-0.3cm}

\bibitem{G14}
L. Grafakos, Classic Fourier Analysis, third edition,
Graduate Texts in Mathematics 249, Springer, New York, 2014.

\vspace{-0.3cm}

\bibitem{h01}
J. Heinonen,
Lectures on Analysis on Metric Spaces,
Universitext. Springer-Verlag, New York, 2001.

\vspace{-0.3cm}

\bibitem{ins14}
M. Izuki, E. Nakai and Y. Sawano,
Function spaces with variable exponents-an introduction-,
Sci. Math. Jpn. 77 (2014), 187-315.

\vspace{-0.3cm}

\bibitem{ins19}
M. Izuki, T, Noi and Y. Sawano,
The John--Nirenberg inequality in ball Banach function spaces
and application to characterization of BMO,
J. Inequal. Appl. 2019, Paper No. 268, 11 pp.

\vspace{-0.3cm}

\bibitem{is17}
M. Izuki and Y. Sawano,
Characterization of BMO via ball Banach function spaces,
Vestn. St.-Peterbg. Univ. Mat. Mekh. Astron. 4(62) (2017), 78-86.

\vspace{-0.3cm}

\bibitem{jtyyz21}
H. Jia, J. Tao, D. Yang, W. Yuan and Y. Zhang,
Special John--Nirenberg--Campanato
spaces via congruent cubes,
Sci. China Math. 65 (2022), 359-420.

\vspace{-0.3cm}

\bibitem{jtyyz22a}
H. Jia, J. Tao, D. Yang, W. Yuan and Y. Zhang,
Boundedness of Calder\'on--Zygmund operators on
special John--Nirenberg--Campanato and Hardy-type spaces via congruent cubes,
Anal. Math. Phys. 12 (2022), Paper No. 15, 56 pp.

\vspace{-0.3cm}

\bibitem{jtyyz22b}
H. Jia, J. Tao, D. Yang, W. Yuan and Y. Zhang,
Boundedness of fractional integrals on special John--Nirenberg--Campanato
and Hardy-type spaces via congruent cubes,
Front. Math. China (to appear).

\vspace{-0.3cm}

\bibitem{jlx19}
R. Jiang, K. Li and J. Xiao,
Flow with $A_\fz(\rr)$ density and transport equation in $\BMO(\rr)$,
Forum Math. Sigma 7 (2019), e43, 30 pp.

\vspace{-0.3cm}

\bibitem{JN61}
F. John and L. Nirenberg,
On functions of bounded mean oscillation,
Comm. Pure Appl. Math. 14 (1961), 415-426.

\vspace{-0.3cm}

\bibitem{km21}
J. Kinnunen and K. Myyryl\"ainen,
Dyadic John--Nirenberg space,
Proc. Roy. Soc. Edinburgh Sect. A (2021),
https://doi.org/10.1017/prm.2021.66.

\vspace{-0.3cm}

\bibitem{lmv20}
K. Li, H. Martikainen and E. Vuorinen,
Bloom type upper bounds in the product BMO setting,
J. Geom. Anal. 30 (2020), 3181-3203.

\vspace{-0.3cm}

\bibitem{lny18}
W. Li, E. Nakai and D. Yang,
Pointwise multipliers on BMO spaces with non-doubling measures,
Taiwanese J. Math. 22 (2018), 183-203.

\vspace{-0.3cm}

\bibitem{SLF95}
S. Lu, Four Lectures on Real $H^{P}$ Spaces, World Scientific,
Singapore, 1995.

\vspace{-0.3cm}

\bibitem{S65}
G. Stampacchia,
The spaces $L^{(p,\lambda)}$, $N^{(p,\lambda)}$ and interpolation,
Ann. Scuola Norm. Sup. Pisa Cl. Sci. (3) 19 (1965), 443-462.

\vspace{-0.3cm}

\bibitem{S93}
E. M. Stein, Harmonic Analysis: Real-Variable Methods,
Orthogonality, and Oscillatory Integrals,
Princeton Math. Ser. (PMS-43),
Princeton University Press, Princeton, 1993.
\vspace{-0.3cm}

\bibitem{TG80}
M. Taibleson and G. Weiss,
The molecular characterization of certain Hardy spaces, in:
Representation Theorems for Hardy Spaces, Ast\'{e}risque 77,
Soc. Math. France, Pairs, 1980, pp. 67-149.

\vspace{-0.3cm}

\bibitem{tyyM}
J. Tao, D. Yang and W. Yuan,
A survey on function spaces of John--Nirenberg type,
Mathematics 9 (2021), Art. No. 2264, 57 pp, https://doi.org/10.3390/math9182264.

\vspace{-0.3cm}

\bibitem{tyyBJMA}
J. Tao, D. Yang and W. Yuan,
A bridge connecting Lebesgue and Morrey spaces via Riesz norms,
Banach J. Math. Anal. 15 (2021), Paper No. 20, 29 pp.

\vspace{-0.3cm}

\bibitem{tyyNA}
J. Tao, D. Yang and W. Yuan,
John--Nirenberg--Campanato spaces,
Nonlinear Anal. 189 (2019), 111584, 36 pp.

\vspace{-0.3cm}

\bibitem{WS22}
D. Wang and L. Shu,
New function classes of Morrey--Campanato type and
their applications,
Banach J. Math. Anal. 16 (2022), Paper No. 36, 31 pp.

\vspace{-0.3cm}

\bibitem{ZT2021}
Z. Zeng, D.-C. Chang, J. Tao and D. Yang, Nontriviality of Riesz--Morrey spaces, Appl. Anal. (2021), https://doi.org/10.1080/00036811.2021.1932836.

\vspace{-0.3cm}
\end{thebibliography}
\end{document}